\def\misajour{01/06//2009}    %

\documentclass{svmultA}

\usepackage[applemac]{inputenc}

\usepackage{makeidx}         
\usepackage{graphicx}        
\usepackage{multicol}        
\usepackage[bottom]{footmisc}
\usepackage{url}
\usepackage{amsmath}

\usepackage{latexsym}
\usepackage{amssymb}
\usepackage[english]{babel} 
\usepackage[hypertex,
colorlinks,bookmarks=true]{hyperref}

\makeindex             

\def\omegahat{\widehat{\omega}}
\def\lambdahat{\widehat{\lambda}}
\def\what{\widehat{w}}

\def\rmH{{\mathrm {H}}}
\def\rmL{{\mathrm {L}}}

\def\cprime{$'$}

\def\virgule{\raise 2pt \hbox{,}}

\newcommand{\m}[1]{${#1}$}

\newcommand{\M}[1]{$${#1}$$}

\def\bC{\mathbf{C}}

\def\bQ{\mathbf{Q}}

\def\bR{\mathbf{R}}

\def\bZ{\mathbf{Z}}
 
 \def\uu{\underline{u}}
 \def\uv{\underline{v}}
 
 \def\Qbar{\overline{\bQ}}

\def\bP{\mathbf{P}}

\newtheorem{thm}[equation] {Theorem} 
\newtheorem{lemme}[equation]{Lemma}  

\newtheorem{cor}[equation]{Corollary}

\newtheorem{conj}[equation]{Conjecture}

\newtheorem{prop}[equation]{Proposition}

\newtheorem{Pb}[equation]{Open Problem}

\newcommand{\ve}{\varepsilon}

\renewcommand{\le}{\leqslant}
\renewcommand{\ge}{\geqslant}

\newcommand{\N}{\mathbb{N}}
\newcommand{\Z}{\mathbb{Z}}

\newcommand{\R}{\mathbb{R}}
\newcommand{\C}{\mathbb{C}}

\newcommand{\cA}{\mathcal{A}}

\begin{document}

\title*{Report on some recent advances in Diophantine approximation}
 \titlerunning{Recent advances in Diophantine approximation} 
 
\author{Michel Waldschmidt
}
 \authorrunning{Michel Waldschmidt}
 
\institute{Université Pierre et Marie Curie--Paris 6, 
UMR 7586 IMJ Institut de Mathématiques de Jussieu, 
175 rue du Chevaleret,
Paris, F--75013 France
 \texttt{\href{mailto:miw@math.jussieu.fr}{miw@math.jussieu.fr}}
\\
 \texttt{\href{http://www.math.jussieu.fr/~miw/}%
{http{:}//www.math.jussieu.fr/$\sim$miw/}   \hfill   update: \misajour}
}

\maketitle


\tableofcontents


\medskip

\goodbreak

\section*{Acknowledgement}

Many thanks to 
Boris Adamczewski, 
Victor Beresnevich, 
Yann Bugeaud, 
Maurice Dodson, 
Michel Laurent,
Claude Levesque,
Damien Roy
for their enlightening remarks and their comments on preliminary versions of this paper.  
Sections 2.7 and 3.6, as well as part of section 1.2, have been written by Victor Beresnevich and Maurice Dodson. 
I wish also to thank Dinakar Ramakrishnan who completed the editorial work in a very efficient way.

\section*{Abstract}\label{S:Abstract}

A basic question of Diophantine approximation, which is the first issue we discuss, is to investigate the rational approximations to a single real number. Next, we consider the algebraic or polynomial approximations to a single complex number, as well as the  simultaneous approximation of powers of a real number by rational numbers with the same denominator. Finally we study generalisations of these questions to higher dimensions. 
Several recent advances have been made by B.~Adamczewski, Y.~Bugeaud,  S.~Fischler,  M.~Laurent, T.~Rivoal, D.~Roy and  W.M.~Schmidt, among others. We review some of these works.

\section*{Introduction}\label{S:Introduction}

The  history of Diophantine approximation is quite old: it includes, for instance, early estimates for $\pi$, computations related to astronomical studies,  the theory of continued fraction expansion. 

There are positive results: {\it any irrational number has good rational approximations}. One of the simplest  tools to see this  is Dirichlet's box principle, other methods are continued fraction expansions, Farey series, geometry of numbers (Minkowski's Theorem). There are negative results: {\it no number has too good (and at the same time too frequent) approximations}. Some results are valid for all (irrational) numbers, others only for restricted classes of numbers, like the class of algebraic numbers.  There is a metric theory (\S\ref{SS:RAAN-AURAMR}) which deals with almost all numbers in the sense of the Lebesgue measure.

   One main goal of the theory of Diophantine approximation is to compare, on the one hand, the distance between a given real number $\xi$ and a rational number $p/q$,  with, on the other hand, the denominator $q$ of the approximant. An approximation is considered as {\it sharp} if $|\xi-p/q|$ is {\it small} compared to $q$. 
This subject is  a classical one, there are a number of surveys, including those by S.~Lang
\cite{MR33:1286,MR44:6615,MR50:12914,MR90k:11032}. Further general references are \cite{MR0087708,MR22:2598,MR1451873,MR0214551,0421.10019,MR99a:11088b,MR1727177,MR2136100}.

   The  works by J.~Liouville, A.~Thue, C.L.~Siegel, F.J.~Dyson, A.O.~Gel'fond, Th.~Schneider and K.F.~Roth essentially solve the question for the case where $\xi$ is algebraic. In a different direction, a lot of results are known which are valid for almost all numbers, after Khintchine and others. 
      
   Several questions arise in this context. One may consider either {\it asymptotic} or else {\it uniform} approximation. The former only asks for infinitely many solutions to some inequality, the latter requires that occurrences of such approximations are not too lacunary. As a consequence, one introduces in \S~\ref{SS:RAAN-AURA} two exponents for the rational approximation to a single real number $\xi$, namely $\omega(\xi)$ for the asymptotic approximation  and $\omegahat(\xi)$ for the uniform approximation; a lower bound for such an exponent means that sharp rational approximations exist, an upper bound means that  too sharp estimates do no exist. To indicate with a ``hat'' the exponents of {\it uniform} Diophantine approximation is a convention which originates in \cite{MR2149403}. 
       
In this context a new exponent, $\nu(\xi)$, inspired by the pioneer work of R.~Apéry in 1976 on $\zeta(3)$,  has been introduced recently by  T.~Rivoal and S.~Fischler (\S~\ref{SS:RAAN-EFR}).

\medskip
   
After rational approximation to a single real number, several other questions naturally arise. One may investigate, for instance, the  {\it algebraic} approximation properties of real or complex numbers, replacing the set of rational numbers by the set of real or complex algebraic numbers. Again, in this context, there are two main points of view: either one considers the distance $|\xi-\alpha|$   between the given real or complex number $\xi$ and algebraic numbers $\alpha$, or else one investigates the smallness of $|P(\xi)|$ for $P$ a non--zero polynomial with integer coefficients. In both cases there are two parameters,  the degree and the height of the algebraic number or of the polynomial,  in place of a single one in degree $1$, namely $q$ for $\xi-p/q$ or for $P(X)=qX-p$. Algebraic and polynomial approximations are related: on the one hand (Lemma~\ref{L:EasyPart}), the irreducible polynomial of an algebraic number close to $\xi$ takes a small value at $\xi$, while on the other hand (Lemma~\ref{L:XimoinsGamma}),  a polynomial taking a small value at $\xi$ is likely to have a root close to $\xi$.  However these connections are not completely understood yet: for instance, while it is easy (by means of Dirichlet's box principle -- Lemma~\ref{L:BoxPrinciplePolynomes}) to prove the existence of polynomials $P$ having small values $|P(\xi)|$ at the point $\xi$, it is not so easy to show that sharp algebraic approximations exist (cf. Wirsing's Conjecture~\ref{C:Wirsing}).

The occurrence of two parameters raises more questions to investigate: often one starts by taking the degree fixed and looking at the behaviour of the approximations as the height tends to infinity; one might do the opposite, fix the height and let the degree tend to infinity:
this is is the starting point of a classification of complex numbers  by V.G.~Sprind{\v{z}}uk (see  \cite{MR2136100} Chap.~8 p.~166). Another option is to let the sum of the degree and the logarithm of the height tend to infinity: this is the choice of S.~Lang who introduced the notion of {\it size} \cite{MR35:5397} Chap.~V in connection with questions of algebraic independence \cite{MR33:1286}. 

The  approximation properties of a real or complex number $\xi$ by polynomials of degree at most $n$ (\S~\ref{SS:PASASN-PACN})  will give rise to two exponents, $\omega_n(\xi)$ and $\omegahat_n(\xi)$, which coincide with $\omega(\xi)$    and $\omegahat(\xi)$ for $n=1$. 
Gel'fond's  Transcendence Criterion  (\S~\ref{SS:PASASN-GTC})
is related to an upper bound for the asymptotic exponent of polynomial approximation $\omegahat_n(\xi)$ valid for all transcendental 
numbers. 

The  approximation properties of a real or complex number $\xi$ by algebraic numbers of degree at most $n$ (\S~\ref{SS:PASASN-AACN})  will give rise to two further exponents, an asymptotic one $\omega^*_n(\xi)$ and and a uniform one $\omegahat^*_n(\xi)$, which  also coincide with $\omega(\xi)$    and $\omegahat(\xi)$ for $n=1$. 

For a {\emph {real}} number $\xi$, there is a third way of extending the investigation of rational approximation, which is the study of simultaneous approximation by rational numbers of the $n$-tuple $(\xi,\xi^2,\ldots,\xi^n)$. Once more there is  an asymptotic exponent $\omega'_n(\xi)$ and a uniform one $\omegahat'_n(\xi)$, again they coincide with $\omega(\xi)$    and $\omegahat(\xi)$ for $n=1$. These two new exponents suffice to describe the approximation properties of a real number by algebraic numbers of degree at most $n$ (the star exponents), thanks to a  transference result (Proposition~\ref{P:Polar}) based on the theory of convex bodies of Mahler. 

Several relations among these exponents are known, but a number of problems remain open: for instance, for fixed  $n\ge 1$ the spectrum of the sextuple 
$$
(\omega_n(\xi),\; \omegahat_n(\xi),\; 
\omega'_n(\xi),\; \omegahat'_n(\xi),\; 
\omega^*_n(\xi),\; \omegahat^*_n(\xi))\in\bR^6
$$  
is far from being completely understood. As we shall see for almost all real numbers $\xi$ and for all algebraic numbers of degree $>n$ 
$$
\omega_n(\xi)=\omegahat_n(\xi)=\omega^*_n(\xi)=\omegahat^*_n(\xi)=
n\quad\text{and}\quad
\omega'_n(\xi)=\omegahat'_n(\xi)=1/n.
$$
A review of the known properties of these six exponents is given in  \cite{MR2149403}. We shall repeat some of these facts here (beware that our notation for $\omega'_n$ and  $\omegahat'_n$   are the $\lambda_n$ and $\lambdahat_n$ from  \cite{MR2149403}, which are the inverse of their  $w'_n$ and  $\what'_n$,  also used by Y.~Bugeaud in \S~3.6 of \cite{MR2136100}  --  here we wish to be compatible with the notation of M.~Laurent in  \cite{LaurentMichel0703146} for the higher dimensional case).

Among a number of new results in this direction, we shall describe those achieved by D.~Roy and others. In particular a number of   results for the case $n=2$ have been recently obtained. 

\medskip

Simultaneous approximations to a tuple of numbers is the next step. The Subspace Theorem of W.M.~Schmidt   (Theorem 1B in Chapter VI of \cite{0421.10019}), which is a powerful generalisation of the Thue-Siegel-Roth Theorem,  deals with the approximation of algebraic numbers. It says that tuples of algebraic numbers do behave like almost all tuples. It is  a  fundamental   tool with  a number of deep consequences     \cite{BiluBourbaki}.   Another point of view is the metrical one, dealing with almost all numbers. Further questions arise which should concern all tuples and these considerations raise many open problems.  We shall report on recent work by M.~Laurent who introduces a collection of new exponents for describing the situation. 

\medskip
 
We discuss these questions mainly  in the case of real  numbers. Most results (so far as they are not related to the density of $\bQ$ into $\bR$)  are valid also for complex numbers with some modifications (however see \cite{BugeaudEvertse2007}), as well and for non--Archimedean valuations, especially  $p$-adic numbers but also (to some extent) for function fields. We make no attempt to be exhaustive, there are a number of related issues which we do not study in detail here -- sometimes we just give a selection of recent references. Among them are

\noindent
$\bullet$ Questions of inhomogeneous approximation.

\noindent
$\bullet$ Littlewood's Conjecture (\cite{MR2136100}, Chap.~10). 

\noindent
$\bullet$ 
Measures of irrationality, transcendence, linear independence, algebraic independence of specific numbers. Effective refinements of Liouville's Theorem are studied in 
\cite{1115.11034}
(see also Chap.~2 of  \cite{MR2136100}). 

\noindent
$\bullet$ Results related to the complexity of the development of irrational algebraic numbers, automata, normality of classical constants (including irrational algebraic numbers) -- the Bourbaki lecture by Yu.~Bilu \cite{BiluBourbaki} on Schmidt's subspace Theorem and its applications describes recent results on this topic and gives further references.

\noindent
$\bullet$  Connection between Diophantine conditions and dynamical systems.

\noindent
$\bullet$ 
Diophantine questions related to Diophantine geometry. Earlier surveys dealing extensively with this issue have been written by S.~Lang. A recent reference on this topic is \cite{MR2465098}.


%

\noindent
$\bullet$ In a preliminary version of the present paper, the list of topics which were not covered included also refined results on Hausdorff dimension, Diophantine approximation of dependent quantities and approximation on manifolds,  hyperbolic manifolds, also the powerful approach initiated by Dani and Margulis, developed by many specialists. We quote here V.V.~Beresnevich,  V.I.~Bernik, H.~Dickinson, M.M.~Dodson, D.Y.~Kleinbock,  {\`E}.I.~Kovalevskaya, G.A.~Margulis, F.~Paulin, S.L.~Velani.  However, thanks to the contribution of Victor Beresnevich and Maurice Dodson who kindly agreed to write sections  
\ref{SS:OverviewMetricalPolynomials}
and 
\ref{SS:FurtherMetricalResults} (and also to contribute by adding remarks, especially  on
 \S~\ref{SS:RAAN-AURAMR}), these topics are no more excluded. 



We discuss only briefly a few questions of algebraic independence; there is much more to say on this matter, especially in connection with Diophantine approximation. Although we quote some recent transcendence criteria as well as criteria for algebraic independence, we do not cover fully the topic (and do not mention criteria for linear independence).

\goodbreak

\section{Rational approximation to a  real number}
\label{S:RAAN}

\subsection{Asymptotic and uniform rational approximation: $\omega$ and $\omegahat$}
\label{SS:RAAN-AURA}

Since $\bQ$ is dense in $\bR$, {\it for any \m{\xi\in\bR} and any  \m{\epsilon>0} there exists \m{b/a\in\bQ} for which 
\M{ 
\left| \xi-\frac{b}{a} \right|< \epsilon.
}
}
Let us write the conclusion
\M{ 
 |a\xi- b|< \epsilon a.
}   
     It is easy to improve this estimate: 
  {\it  Let \m{a\in\bZ_{>0}} and let \m{b} be the nearest integer to \m{a\xi}. Then
  \M{
 | a\xi- b|\le 1/2.
}   
}

A  much stronger estimate is due to Dirichlet (1842) and follows from the box or pigeonhole principle -- see for instance  \cite{MR33:1286},  \cite{MR0214551},  \cite{0421.10019} Chap.~I Th.\ 1A:  

\begin{thm}[Uniform Dirichlet's Theorem] \label{T:DirichletUniforme}
For  {each}  real number  \m{N>1},  there exist  \m{q}  and \m{p} in $\bZ$  with \m{1\le q<N}  such that
$$
 | q\xi-p |< \frac{1}{N}\cdotp
$$
 \end{thm}

 As an immediate consequence (\cite{0421.10019} Chap.~I Cor.~1B):
 
 \begin{cor}[Asymptotic Dirichlet's Theorem] \label{Cor:DirichletAsymptotique}
 If  $\xi\in\bR$ is irrational, then there exist infinitely many  \m{p/q\in\bQ}  for which 
$$
\left| \xi- \frac{p}{q}\right|<\frac{1}{ q^2}\cdotp
$$
  
 \end{cor}
 
Our first concern is to investigate whether it is 
 possible to improve the uniform estimate of Theorem~\ref{T:DirichletUniforme} as well as the asymptotic estimate of Corollary~\ref{Cor:DirichletAsymptotique}.
 
We start with
 Corollary~\ref{Cor:DirichletAsymptotique}. Using either the theory of continued fractions or Farey series, one deduces a slightly stronger statement, proved by A.~Hurwitz in 1891 (Theorem 2F in Chap.~I of \cite{0421.10019}): 
 
 \begin{thm}[Hurwitz]\label{T:Hurwitz}
  For any real irrational number $\xi$,  there exist infinitely many  \m{p/q\in\bQ}  such that  
 $$
\left| \xi- \frac{p}{q}\right|< \frac{1}{\sqrt{5} q^2}\cdotp
$$
\end{thm}

For  the Golden Ratio  $\gamma=(1+\sqrt{5})/2$ and for the numbers related to the Golden Ratio  by a homographic transformation $(ax+b)/(cx+d)$ (where $a$, $b$, $c$, $d$ are rational integers satisfying $ad-bc=\pm 1$),  this asymptotic result 
  is optimal.  For all other  irrational real numbers, Hurwitz proved that the constant \m{\sqrt{5}} can be replaced by   \m{\sqrt{8}} (\cite{0421.10019} Chap.~I Cor.~6C).  
These are the first elements of  the
   {\it Lagrange spectrum}: $\sqrt{5}$, $\sqrt{8}$,  $\sqrt{221}/5$, $\sqrt{1517}/13$, \dots \quad (references are given in Chap.~I \S~6 of \cite{0421.10019}; the book \cite{MR1010419} is devoted to the Lagrange and Markoff spectra).

 Lagrange noticed as early as 1767 (see \cite{MR99a:11088b} Chap.~1 Theorem 1.2) that for all irrational quadratic numbers, the exponent $2$ in $q^2$ in the conclusion of Corollary~\ref{Cor:DirichletAsymptotique} is optimal: more generally, Liouville's inequality (1844) produces, for each algebraic number $\xi$ of degree $d\ge 2$,  a  constant $c(\xi)$ such that, for all rational numbers $p/q$,
 $$
\left| \xi- \frac{p}{q}\right|> \frac{c(\xi)}{q^d}\cdotp
$$
Admissible values for $c(\xi)$ are easy to specify (Th.\ 1 Chap.1 \S~1 of \cite{MR19:252f}, \cite{MR0214551} p.~6, Th.\ 1E of \cite{MR1176315}, Th.\ 1.2 of \cite{MR2136100}). 

A {\it Liouville number} is a real number $\xi$ for which the opposite estimate holds:  {\it for any $\kappa>0$, there exists a rational number $p/q$ such that}
\begin{equation}\label{E:Liouville}
0<\left| \xi- \frac{p}{q}\right|< \frac{1}{q^\kappa}\cdotp
\end{equation}
A {\it very well approximable number} is a real number $\xi$  for which there exists $\kappa>2$ such that  the inequality (\ref{E:Liouville}) has infinitely many solutions. 
A nice example of such a number
 is 
$$
\xi_\kappa:=2\sum_{n=1}^\infty 3^{-\lceil \kappa^n\rceil}.
$$ 
for $\kappa$ a real number $>2$.    
This number belongs to the  middle third  
Cantor set ${\mathcal K}$, which is the set of real numbers whose base three expansion are free of the digit $1$. 
In \cite{LevesleySalpVelani}, 
J.~Levesley, C.~Salp  and S.L.~Velani
 show that $\xi_\kappa$  is an element of ${\mathcal K}$ with irrationality exponent $\mu(\xi_{\kappa})= \kappa$ for $\kappa \ge (3+\sqrt{5})/2$ and $\ge \kappa$ for 
 $2<\kappa \le (3+\sqrt{5})/2$.  This example answers a question of K.~Mahler on the existence of very well approximable numbers which are not Liouville numbers in ${\mathcal K}$.  In \cite{pre05284763},  Y.~Bugeaud
shows that $ \mu(\xi_{\kappa})= \kappa$ for $\kappa \ge 2$, and more generally, that for $\kappa \ge 2$ and $\lambda>0$, the number
$$
\xi_{\lambda,\kappa}:= 2\sum_{n=n_0}^\infty 3^{-\lceil \lambda \kappa^n\rceil}
$$ 
has  $\mu(\xi_{\lambda,\kappa})= \kappa$.
%

Given $\kappa\ge 2$, denote by $E_\kappa$ the set of  real numbers $\xi$ satisfying the following property: 
 {\it   the inequality (\ref{E:Liouville}) has infinitely many solutions  in integers $p$ and $q$  and for any 
$c<1$  there exists $q_0$ such that, for $q\ge q_0$,}
\begin{equation}\label{E:exposantdetrdce}
\left| \xi- \frac{p}{q}\right|>\frac{c}{q^\kappa}\cdotp
\end{equation}
Then for any $\kappa\ge 2$ this set $E_\kappa$ is not empty. Explicit examples have been given by Jarn\'\i k in 1931
(see \cite{1126.11036}
for a variant). 
In \cite{MR1866488}, V.V.~Beresnevich, H.~Dickinson and S.L.~Velani 
  raised the question of the Hausdorff dimension of   the set $ E_\kappa$. The answer is given by Y.~Bugeaud in \cite{MR2006007}: this dimension is $2/\kappa$.


We now consider the uniform estimate of Theorem~\ref{T:DirichletUniforme}.
Let us show that for any irrational number $\xi$, Dirichlet's Theorem   is  essentially optimal: one cannot replace $1/N$ by $1/(2N)$. This was already observed by 
Khintchine in 1926
 \cite{JFM52.0183.01}:
 
    \begin{lemme} \label{L:Khinchin}
{\it  Let  $\xi$  be a real number. Assume that  there exists a positive integer $N_0$ such that,  for each   integer \m{N\ge N_0},  there exist  \m{a\in\bZ}  and \m{b\in\bZ}  with \m{1\le a<N}  and
\M{
 | a\xi-b |< \frac{1}{2N}\cdotp
 }
 Then $\xi$ is rational and \m{a\xi=b} for each \m{N\ge N_0}.
 }
  \end{lemme}

  \begin{proof} 
By assumption for each  integer \m{N\ge N_0}  there exist  \m{a_N\in\bZ}  and \m{b_N\in\bZ}  with \m{1\le a_N<N}  and
$$
 | a_N\xi-b_N |< \frac{1}{2N}\cdotp
 $$
Our goal is to check \m{a_N\xi=b_N} for each \m{N\ge N_0}.

   Let \m{N\ge N_0}. Write \m{(a,b)} for  \m{(a_N,b_N)} and  \m{(a',b')} for  \m{(a_{N+1},b_{N+1})}:
   \M{
 | a\xi-b |< \frac{1}{2N} \; (1\le a\le N-1), \quad
  | a'\xi-b' |< \frac{1}{2N+2}
   \; (1\le a'\le N).
  }
 Eliminate $\xi$ between $a\xi-b$ and $a'\xi-b'$:   the rational integer 
\M{
ab'-a'b=a(b'-a'\xi)+a'(a\xi-b)
} 
   satisfies
\m{
|ab'-a'b|<1
},
hence it  vanishes and  \m{ab'=a'b}.

   Therefore the rational number  \m{b_N/a_N=b_{N+1}/a_{N+1}} does not depend on \m{N\ge N_0}. Since
    \M{
 \lim_{N\rightarrow\infty }{b_N}/{a_N}=\xi,
    } 
    it follows that   \m{\xi=b_N/a_N} for all \m{N\ge N_0}.
  \hfill  \hbox{\qed} \null
    
\end{proof}

\noindent
{\bf Remark.} {\it As pointed out to me by M.~Laurent,   an alternative argument 
is based on    continued fraction expansions. }

  \medskip
 
Coming back to  Theorem~\ref{T:DirichletUniforme}  and Corollary~\ref{Cor:DirichletAsymptotique}, we associate to each real irrational numbers $\xi$  two exponents $\omega$ and $\omegahat$ as follows. 

Starting with Corollary~\ref{Cor:DirichletAsymptotique}, we introduce the {\it asymptotic  irrationality exponent  of  a real number $\xi$} which is denoted by $\omega(\xi) $: 
\begin{multline}\notag
\omega(\xi) =   \sup\Big\{w ;  \;   \text{there exist  infinitely many} \;     (p,q) \in \bZ^2\\
 \text{with}\; 
 q\ge 1\; \text{and}\;  0<| q \xi -p |  \le q^{-w}\Big\}.
\end{multline}
 Some authors prefer to introduce the {\it irrationality  exponent $\mu(\xi)=\omega(\xi) +1$ of  $\xi$} which is denoted by $\mu(\xi)$: 
\begin{multline}\notag
\mu(\xi)=  \sup\Big\{\mu ;  \; \text{there exist  infinitely many}     \;    (p,q) \in \bZ^2\\
 \text{with}\;  q\ge 1
 \; \text{and}\; 
0< \left|  \xi- 
\displaystyle \frac{p}{q} \right |  \le q^{-\mu}\Big\}.
\end{multline}
An upper bound for \m{\omega(\xi) } or $\mu(\xi)$ is an {\it irrationality measure} for $\xi$, namely a lower bound for 
\m{|\xi-p/q|} when \m{p/q\in\bQ}. 

Liouville numbers are the real numbers $\xi$ with $\omega(\xi)=\mu(\xi)= \infty$. 

Since no set $E_\kappa$  (see property~(\ref{E:exposantdetrdce})) with $\kappa\ge 2$ is empty,  the {\it spectrum} $\bigl\{ \omega(\xi)\; ;\; \xi\in\bR\setminus\bQ\bigr\}$ of $\omega$ is the whole interval $[1,+\infty]$, while the spectrum  $\bigl\{ \mu(\xi)\; ;\; \xi\in\bR\setminus\bQ\bigr\}$  of $\mu$ is $[2,+\infty]$.

According to the Theorem of Thue--Siegel--Roth \cite{BiluBourbaki}, {\it  for   any real algebraic number \m{\xi\in\bR\setminus\bQ}, }
$$
\omega(\xi)  = 1.
$$
We shall see (in \S~\ref{SS:RAAN-AURAMR}) that the same holds for almost all real numbers.

The other exponent related to Dirichlet's Theorem~\ref{T:DirichletUniforme} is the {\it uniform irrationality exponent of  $\xi$,} denoted by $\omegahat(\xi) $:
\begin{align}\notag
\omegahat(\xi) =     \sup\Big\{ w ;  \;  \text{for any $N   \ge 1$,     there exists $(p,q) \in \bZ^2$} \;
\hskip 2 true cm
\\
\notag
\hfill
 \text{with}\;  1 \le  q \le N \;   \text{and}  \; 0<| q \xi -p |  \le N^{-w}\Big\}.
\end{align}
In the singular case of a rational number $\xi$ we set 
$\omega(\xi) =\omegahat(\xi) =0$.
It is plain from the definitions that for any  \m{\xi\in\bR\setminus\bQ}, 
\M{
\omega(\xi) \ge \omegahat(\xi) \ge 1.
}
In fact Lemma~\ref{L:Khinchin} implies  that  
   {\it for any \m{\xi\in\bR\setminus\bQ},  
$   \omegahat(\xi)=1$.}
 Our motivation to introduce a notation for a number which is always equal to $1$ is that it will become non--trivial in more general situations ($\omegahat_n$ in \S~\ref{SS:PASASN-PACN}, $\omegahat'_n$ in \S~\ref{SS:PASASN-SRARN}, $\omegahat^*_n$ in \S~\ref{SS:PASASN-AACN}).

\subsection{Metric results}
\label{SS:RAAN-AURAMR}

The metric theory of Diophantine approximation provides statements which are valid for almost all (real or complex)  numbers, that means for all numbers outside a set of Lebesgue measure $0$. Among many references on this topic, we quote
\cite{MR0245527,
MR548467,
MR1727177,
Harman, 
MR2136100
}. See also \S~\ref{SS:OverviewMetricalPolynomials}
and 
 \S~\ref{SS:FurtherMetricalResults} below.

One of the early results is due to 
Capelli:  {\it  for almost all $\xi\in\bR$,}
$$
\omega(\xi)=\omegahat(\xi) =1
\quad\text{\it and}\quad
\mu(\xi)=2. 
$$
This is one of many instances where irrational algebraic numbers behave like almost all numbers. However one cannot expect that {\it all} statements from Diophantine Approximation which are satisfied by all numbers outside a set of measure $0$ will be satisfied by all irrational algebraic numbers, just because such an intersection of sets of full measure is empty.
As pointed out to me by B.~Adamczewski, S.~Schanuel (quoted by S.~Lang in 
\cite{MR33:1286} p.~184 and \cite{MR96h:11067} Chap.~II \S~2 Th.\ 6) gave a more precise formulation of such a remark as follows. 

Denote by  $\mathcal K$  (like Khintchine) the set of {\it non--increasing}  
functions $ \Psi $ from $ \bR_{\ge 1}  $ to $  \bR_{>0}$. 
Set
$$
{\mathcal K}_c=\left\{ \Psi\in {\mathcal K}\, ;\, 
\sum_{n\ge 1} \Psi(n) \; \text{converges} \right\},
\quad
{\mathcal K}_d=\left\{ \Psi\in {\mathcal K}\, ;\, 
\sum_{n\ge 1} \Psi(n) \;  \text{diverges} \right\}
$$
Hence ${\mathcal K}={\mathcal K}_c\cup{\mathcal K}_d$. 

A well-known  theorem of A.~Ya.~Khintchine in 1924  
 (see  \cite{Khintchine-1924}, \cite{MR1451873} and Th.\ 1.10 in \cite{MR2136100}) has been refined as follows (an extra condition that the function $x\mapsto x^2 \Psi(x)$ is decreasing has been dropped) by  Beresnevich, Dickinson and Velani \cite{Beresnevich-Dickinson-Velani-06:MR2184760}:

\begin{thm}[Khintchine] \label{T:Khinchin}
Let $\Psi \in {\mathcal K}$. Then for almost all real numbers $\xi$, the inequality 
\begin{equation}\label{E:Khinchin}
\left| q\xi-p
\right |<\Psi(q)
\end{equation}
has 
\begin{itemize}
\item 
only finitely many solutions in integers $p$ and $q$ if  $\Psi\in {\mathcal K}_c$ 
\item 
infinitely many solutions  in integers $p$ and $q$  if  $\Psi\in {\mathcal K}_d$.
\end{itemize}

\end{thm}

\def\remarkMMD{
The case in Khintchine's theorem when the sum converges is essentially
the easy half of the Borel-Cantelli Lemma and $\Psi$ does not need to
be decreasing.  The case when the sum diverges is much harder.
Originally proved using continued fractions~\cite{Kh24}, later proofs
involve measure and probabilistic ideas capable of
generalisation~\cite{HarmanMNT,Kh26,Sprindzuk}.  In their
paper~\cite{DuffinSchaeffer}, Duffin and Schaeffer also showed that
some monotonicity condition cannot be dropped in the divergent case by
producing a counter-example. They made the conjecture that when $\Psi$
is not necessarily monotonic, almost all or almost no $\xi\in\R$
satisfy
\begin{equation}
  \label{eq:1}
  |q\xi-p|< \Psi(q)
\end{equation}
for infinitely many integers 
$p,q$ with $(p,q)=1$, according as the sum
\begin{equation}
  \label{eq:2}
  \sum_{n\ge 1} \Psi(n)\frac{\varphi(n)}{n},
\end{equation}
where $\varphi$ is the Euler phi function, diverges or
converges~\cite{DuffinSchaeffer}.  Erd\"os~\cite{Erdos70} and
Vaaler~\cite{vaaler78} made some progress but the conjecture is still
unproved.  When coprimality is dropped, Catlin~\cite{catlin76} has
proposed a more complicated sum, also involving the Euler phi
function.

}
 
S.~Schanuel proved  that the set of real numbers which behave like almost all numbers from the point of view of Khintchine's Theorem in the convergent case has measure $0$. More precisely the set 
 of real numbers $\xi$ such that, for any {\it smooth convex} function $\Psi\in {\mathcal K}_c$, the inequality (\ref{E:Khinchin}) has only finitely many solutions, is the set of real numbers with bounded partial quotients  ({\it badly approximable numbers} -- see \cite{0421.10019} Chap.~I \S~5;  other characterisations of this set are given in \cite{MR96h:11067} Chap.~II \S~2 Th.\ 6).

 
 
 Moreover B. Adamczewski and Y. Bugeaud noticed that given any
irrational $\xi$, either there exists a $\Psi\in \mathcal{K}_d$ for
which
$$
|q\xi-p|<\Psi(q)
$$ has no integer solutions or there exists a $\Psi\in \mathcal{K}_c$ for
which
$$
|q\xi-p|<\Psi(q)
$$ has infinitely many integer solutions.

\subsection{The exponent $\nu$ of S.~Fischler and T.~Rivoal}
\label{SS:RAAN-EFR}

    Let \m{\xi\in\bR\setminus\bQ}. In \cite{FischlerRivoal},  S.~Fischler and T.~Rivoal introduce a new exponent $\nu(\xi) $ which they define as follows.

When \m{\uu=(u_n)_{n\ge 1}} is an increasing sequence of positive integers, define another sequence of integers  \m{\uv=(v_n)_{n\ge 1}} by \m{|u_n\xi-v_n|<1/2} (i.e{.} $v_n$ is the nearest integer to $u_n\xi$) and set 
\M{
\alpha_\xi(\uu)=\limsup_{n\rightarrow\infty} \frac{
|u_{n+1}\xi - v_{n+1}|}{|u_n\xi - v_n|}\virgule
\quad\beta(\uu)=\limsup_{n\rightarrow\infty} \frac{
u_{n+1}}{u_n}\cdotp
}
Then
\M{
\nu(\xi)= \inf  \log \sqrt{\alpha_\xi(\uu)\beta(\uu)},} 
where \m{\uu} ranges over the sequences which satisfy \m{\alpha_\xi(\uu)<1} and \m{\beta(\uu)<+\infty.}
Here we agree that \m{\inf {\emptyset}=+\infty}.

They establish a connection with the irrationality exponent by proving:
\M{
\mu(\xi)\le 1- \frac{ \log\beta(\uu)}{
\log\alpha_\xi(\uu)
}\cdotp
}
As a consequence,  {\it if \m{\nu(\xi)<+\infty}, then \m{\mu(\xi)<+\infty}. }

If $\xi$ is quadratic, Fischler and Rivoal produce a sequence \m{\uu} with \m{\alpha_\xi(\uu)\beta(\uu)=1}, hence \m{\nu(\xi)=0}.

This new exponent $\nu$  is motivated by Apéry-like proofs of irrationality  and  measures. 
Following the works of R.~Ap\'ery, A.~Baker, F.~Beukers, G.~Rhin and C.~Viola, M.~Hata among others,  S.~Fischler and  T.~Rivoal  deduce
\M{
\nu(2^{1/3})\le (3/2)\log 2,\quad
\nu(\zeta(3) ) \le 3,\quad 
\nu(\pi^2)\le 2,\quad
\nu(\log 2)\le 1.
}
Also \m{\nu(\pi)\le 21}.

 The spectrum of \m{\nu(\xi)} is not yet known. According to  \cite{FischlerRivoal}, {\it for any \m{\xi\in\bR\setminus\bQ}, the inequalities 
   \m{0\le \nu(\xi)\le +\infty} hold. } Further, {\it for almost all  \m{\xi\in\bR},  \m{ \nu(\xi)=0}.} Furthermore,  S.~Fischler and T.~Rivoal, completed by B.~Adamczewski \cite{AdamczewskiExposantDensite}, proved that 
{\it      any irrational algebraic real number $\xi$ has \m{\nu(\xi)<+\infty}. }

There are examples of \m{\xi\in\bR\setminus\bQ}  for which \m{\nu(\xi)=+\infty}, but 
all known examples with  \m{\nu(\xi)=+\infty} so far have \m{\mu(\xi)=+\infty}.

 Fischler and Rivoal ask whether 
{\it it is true that \m{\nu(\xi)<+\infty} implies \m{\mu(\xi)=2}.}
Another related question they raise in \cite{FischlerRivoal}  is whether {\it 
 there are numbers $\xi$ with \m{0<\nu(\xi)<+\infty}.}

\section{Polynomial, algebraic and simultaneous approximation to a  single number}
\label{S:PASASN}

We define the (usual) height $\rmH(P)$ of a polynomial 
$$
P(X)=a_0+a_1X+\cdots+a_nX^n
$$
 with complex coefficients as the maximum modulus of its coefficients, 
 while its length $\rmL(P)$ is the sum of the moduli of these coefficients:
 $$
 \rmH(P)=\max_{0\le i\le n} |a_i|,
 \quad
 \rmL(P)=\sum_{i=0}^n |a_i|.
  $$ 
 The height $\rmH(\alpha)$ and length $\rmL(\alpha)$  of an algebraic number $\alpha$ are the height and length of its minimal polynomial over $\bZ$.

\subsection{Connections between polynomial approximation 
and approximation by algebraic numbers}
\label{SS:PASASN-PASAAAN}

 Let $\xi$ be a complex number.  To produce a sharp {\it  polynomial approximation } to $\xi$ is  to find a non--zero polynomial \m{P\in\bZ[X]}
 for which 
   \m{|P(\xi)|} is small. An {\it algebraic approximation} to $\xi$ is an algebraic number $\alpha$ such that the distance   \m{|\xi-\alpha|} between $\xi$ and  $\alpha$ is small. There are close connections between both questions. On the one hand, if  \m{|P(\xi)|} is small, then $\xi$ is close to a root \m{\alpha} of \m{P}. On the other hand, if  \m{|\xi-\alpha|} is small then the minimal polynomial of \m{\alpha} assumes a small value at $\xi$. These connections explain that the classifications of transcendental numbers in  $S$, $T$ and $U$ classes by K.~Mahler coincide with the  classifications of transcendental numbers in  $S^*$, $T^*$ and $U^*$ classes by J.F.~Koksma (see \cite{MR19:252f}  Chap.~III  and \cite{MR2136100} Chap.~3).

The easy part is the next statement (Lemma 15 Chap.~III \S~3 of \cite{MR19:252f},
 \S~15.2.4 of \cite{MR1756786}), Prop.~3.2  \S~3.4 of \cite{MR2136100}.

       
\begin{lemme}\label{L:EasyPart}
Let $f\in\bC[X]$ be a non--zero polynomial of degree $D$ and length $L$, let $\alpha\in\bC$ be a root of $f$ and let $\xi\in\bC$ satisfy $|\xi-\alpha|\le 1$. Then
$$
|f(\xi)| \le |\xi-\alpha|LD(1+|\xi|)^{D-1}.
$$

\end{lemme}   

The other direction requires more work (see    Chap.~III \S~3 of \cite{MR19:252f},   \S~3.4 of \cite{MR2136100}).  
 The next result is due to G.~Diaz and M.~Mignotte \cite{MR1124725} (cf. Lemma 15.13 of  \cite{MR1756786}).
 

\begin{lemme}  \label{L:XimoinsGamma}
Let $f\in\bZ[X]$ be a non--zero polynomial of
degree $D$. Let $\xi$ be a complex number, $\alpha$ a root
of $f$ at minimal distance of $\xi$ and $k$ the multiplicity
of $\alpha$ as a root of $f$. Then
$$
|\xi-\alpha|^k\le D^{3D-2} \rmH(f)^{2D}|f(\xi)|.
$$

\end{lemme}

 Further similar estimates are due to M.~Amou and Y.~Bugeaud 
 \cite{MR2428209}.

\subsection{Gel'fond's Transcendence Criterion}
\label{SS:PASASN-GTC}


The so--called {\it Transcendence Criterion, } proved by A.O.~Gel'fond in 1949,  is an auxiliary result in the method he introduced in  \cite{MR10:682d,MR11:83b,MR11:231d} (see also \cite{MR13:727b} and \cite{MR22:2598}) for proving  algebraic independence results. An example is the algebraic independence  of the two numbers  \m{2^{\root 3 \of 2}}  and \m{2^{\root 3 \of 4}}.    
More generally, he  proved that {\it  if $\alpha$ is a non--zero algebraic number, $\log\alpha$ a non--zero logarithm of $\alpha$ and  \m{\beta} an algebraic number of degree \m{d\ge 3}, then at least \m{2} among the  \m{d-1} numbers
   \M{
\alpha^{\beta}, \; \alpha^{\beta^2}, \; \dots \, , \alpha^{\beta^{d-1}
}
}
are algebraically independent.
}
Here $\alpha^z$ stands for $\exp(z\log\alpha)$. 

While Gel'fond--Schneider  transcendence method for solving Hilbert's seventh problem on the transcendence of $\alpha^\beta$ relies on a {\it Liouville type} estimate, namely a lower bound for a non--zero value \m{|P(\xi)|} of a polynomial $P$ at an algebraic point $\xi$, Gel'fond's method for algebraic independence requires a more sophisticated result, namely the fact that {\it there is no non--trivial uniform sequence of polynomials taking small values at a given transcendental number. }

 Here is a  version of this Transcendence Criterion
\cite{MR45:3333,
MR2000b:11088}.

\begin{thm}[Gel'fond's Transcendence Criterion] \label{T:GTC}
  Let  \m{\xi\in\bC}. Assume there is a sequence  \m{(P_N)_{N\ge N_0}} of non--zero polynomials in   \m{\bZ[X]}, where   \m{P_N} has degree  \m{\le N} and  height  \m{\rmH(P_N) \le e^N}, for which 
 \M{
|P_N(\xi)|\le e^{-6N^2}.
}
Then  $\xi$ is algebraic and  \m{P_N(\xi)=0} for all  \m{N\ge N_0}. 
\end{thm}
 

\noindent 
\begin{proof}[sketch of]
The idea of the proof is basically the same as for Lemma~\ref{L:Khinchin} which was dealing with degree $1$ polynomials: one eliminates the variable using two consecutive elements of the sequence of polynomials. In degree $1$ linear algebra was sufficient. For higher degree the resultant of polynomials is a convenient substitute. 

Fix \m{N\ge N_0}. Since \m{|P_N(\xi)|} is small, $\xi$ is close to a root \m{\alpha_N} of \m{P_N}, hence \m{P_N} is divisible by a power \m{Q_N} of the irreducible polynomial of \m{\alpha_N} and \m{|Q_N(\xi)|} is small. The resultant of the two polynomials  \m{Q_N} and \m{Q_{N+1}} has absolute value \m{<1}, hence it vanishes and therefore \m{\alpha_N} does not depend on \m{N}. 
  \hfill  \hbox{\qed} \null

\end{proof}

  In  1969,  H.~Davenport and W.M.~Schmidt (\cite{MR0246822} Theorem 2b) prove the next  variant of Gel'fond's Transcendence Criterion, where now the degree is fixed. 
   
\begin{thm}[Davenport and Schmidt] \label{T:DavenportSchmidtCriterion}
      Let $\xi$ be a real number and \m{n\ge 2} a positive integer.
    Assume that for each sufficiently large positive integer \m{N} 
 there exists a non--zero polynomial \m{P_N\in\bZ[X]} of degree \m{\le n} and usual height \m{\le N} for which 
    \M{
    |P_N(\alpha)|\le N^{-2 n+1}.
    }
Then $\xi$ is algebraic of degree $\le n$. 
\end{thm}



The next  sharp version of Gel'fond's Transcendence Criterion \ref{T:GTC}, restricted to  quadratic polynomials, is due to B.~Arbour and  D.~Roy   2004  \cite{MR2038064}.

\begin{thm}[Arbour and  Roy]\label{Th:ArbourRoy}
 Let $\xi$ be a complex number.  Assume that there exists $N_0>0$ such that, for any  \m{N\ge N_0}, there exists a polynomial \m{P_N\in\bZ[X]} of degree \m{\le 2} and height \m{\le N} satisfying
  \M{
  |P_N(\xi)|\le \frac{1}{4} N^{-\gamma-1}.
  }
 Then $\xi$ is algebraic of degree $\le 2$ and   \m{P_N(\xi)=0}  for all $N\ge N_0$. 
  
  \end{thm}
  
 Variants of the transcendence criterion have been considered by D.~Roy in connection with his new approach towards Schanuel's Conjecture   \cite{MR1756786} 
\S~15.5.3:

 \begin{conj}[Schanuel] \label{C:Schanuel}
    Let \m{x_1,\ldots,x_n} be $\bQ$-linearly independent complex numbers. Then \m{n} at least of the \m{2n} numbers  \m{x_1,\ldots,x_n,\; e^{x_1}, \ldots, e^{x_n}} are algebraically independent.
    \end{conj}

In  \cite{MR2001m:11120,MR2002e:32017}, D.~Roy states a Diophantine approximation conjecture, which he shows to be equivalent to Schanuel's Conjecture \ref{C:Schanuel}.


Roy's Conjecture involves polynomials for which bounds for the degree, height and absolute values at given points are assumed. The first main difference with Gel'fond's situation is that the smallness of the values is not as strong as what is achieved by Dirichlet's box principle when a single polynomial is constructed. Hence elimination arguments cannot be used without further ideas. On the other hand the assumptions involve not only one polynomial for each $n$ as in Theorem \ref{T:GTC}, but a collection of polynomials, and they are strongly related. This new situation raises challenging questions on which some advances has already been achieved. In particular in \cite{MR2000b:11088} Laurent and Roy obtain variants of Gel'fond's Criterion \ref{T:GTC} involving multiplicities. Further progress has been subsequently made by D.~Roy in 
\cite{RoyDamienSmallValueAdditive,RoyDamienSmallValueMultplicative}. 

\smallskip

We shall consider extensions of the Transcendence Criterion  to criteria for algebraic independence in  \S~\ref{SS:SDAHD-CAI}.

\subsection{Polynomial approximation to a single number
}
\label{SS:PASASN-PACN}

  A simple application of Dirichlet's box principle (see the proof of  Lemma 8.1 in 
\cite{MR2136100}) yields the existence of polynomials with small values at a given real point:

   \begin{lemme}  \label{L:BoxPrinciplePolynomes}
Let $\xi$ be a real number and $n$ a positive integer.   
Set   $c=(n+1)\max\{1,|\xi|\}^n$. Then, for each positive integer  \m{N}, there exists a non--zero polynomial \m{P\in\bZ[X]}, of degree \m{\le n} and usual height \m{\rmH(P) \le N}, satisfying
    \M{
    |P(\xi)|\le c N^{- n}.
    }
\end{lemme}  

Variants of this lemma rely on the geometry of numbers: for instance from Th.\ B2 in \cite{MR2136100} one deduces  that in the case 
  $0<|\xi|< 1/2$, if $N\ge 2$, then  the conclusion holds also with $c=1$ (see the proof of  Prop. 3.1 in \cite{MR2136100}).

Theorem \ref{T:DirichletUniforme} in  \S\ref{S:RAAN} yields a refined estimate for the special case $n=1$.  The statement is plain in the case where $\xi$ is algebraic of degree $\le n$ as soon as $N$ exceeds the height of the irreducible polynomial of $\xi$ -- this is why when $n$ is fixed we shall most often assume that $\xi$ is either transcendental or else algebraic of degree $>n$.

The fact that the exponent $n$ in Lemma~\ref{L:BoxPrinciplePolynomes}  cannot be replaced by a larger number  (even if we ask only such a solution for infinitely many $N$) was proved by 
Sprind{\v{z}}uk \cite{MR0245527}, who showed in 1965 that {\it  for $\xi$ outside a set of Lebesgue measure zero and for each $\epsilon>0$, there are only finitely many non--zero integer polynomials of degree at most $n$ with 
$$
|P(\xi)|\le \rmH(P)^{-n-\epsilon}.
$$
}

We introduce, for each positive integer $n$ and each real number   $\xi$,
 two exponents $\omega_n(\xi)$ and $\omegahat_n(\xi)$ as follows.

The number  \m{\omega_n(\xi)} denotes the supremum of the real numbers $w$  for which  there exist    {infinitely many} positive integers  \m{N} for which the system of inequalities 
\begin{equation}\label{E:omega_n}
0<   |x_0+x_1\xi+\cdots+x_n\xi^n|\le N^{-w},
   \quad 
  \max_{0\le i\le n} |x_i| \le N
\end{equation}
  has a solution in rational integers \m{x_0,x_1,\ldots,x_n}. The inequalities  (\ref{E:omega_n}) can be written
  $$
0<  |P(\xi)|\le \rmH(P)^{-w}
  $$
where $P$ denotes a non--zero polynomial with integer coefficients and degree $\le n$.    A   {\it transcendence measure} for 
  \m{ \xi} is a lower bound for $|P(\xi)|$ in terms of the height $\rmH(P)$ and the degree $\deg P$ of $P$. 
Hence one can view  an upper bound for   \m{ \omega_n(\xi)}   as a transcendence measure for 
  \m{ \xi}.

These numbers arise in Mahler's classification of complex numbers (\cite{MR19:252f}  Chap.~III \S~1  and \cite{MR2136100} \S~3.1).

A uniform version of  this exponent is  the supremum  \m{\omegahat_n(\xi)}   of the real numbers $w$  such that,  for any sufficiently large  integer  \m{N}, the same  system (\ref{E:omega_n}) has a solution.
   An upper bound for   \m{\omegahat_n(\xi)} is a {\it uniform  transcendence measure} for 
  \m{ \xi}. 
  
Clearly, from the definitions, we see that these exponent generalize the ones from \S\ref{SS:RAAN-AURA}:
   for \m{n=1},  \m{ \omega_1(\xi) =\omega(\xi)} and \m{ \omegahat_1(\xi)=\omegahat(\xi)}.
From Lemma~\ref{L:BoxPrinciplePolynomes} one deduces, 
for any $n\ge 1$ and any $\xi\in\bR$ which is not algebraic of degree $\le n$,
\begin{equation}\label{E:MinorationOmega-n}
n\le  \omegahat_n(\xi)\le \omega_n(\xi).
\end{equation}
Moreover, \m{\omega_n\le \omega_{n+1}} and  \m{\omegahat_n\le \omegahat_{n+1}}. As a consequence, 
Liouville numbers have \m{\omega_n(\xi)=+\infty} for   all \m{n\ge 1}.

 
 The value of the exponents $ \omega_n$ and $ \omegahat_n$ for almost all real numbers and for all algebraic numbers of degree $>n$ is $n$. The following metric result   is due to V.G.~Sprind{\v{z}}uk 
 \cite{MR0245527}:
  
\begin{thm}[Sprind{\v{z}}uk]\label{T:omega_nPresquePartout}
For almost all numbers \m{\xi\in\bR},  
$$
\omega_n(\xi)=\omegahat_n(\xi)=  n \;
\text{for all} \;  n\ge 1.
$$
\end{thm}

As a consequence of W.M.~Schmidt's subspace Theorem one deduces  (see \cite{MR2136100} Th.\ 2.8 and 2.9) the value of  $ \omega_n(\xi)$ and $ \omegahat_n(\xi)$ for $\xi$ algebraic irrational: 
 
\begin{thm}[Schmidt]  \label{T:ConsequenceSubspace}
Let  \m{n\ge 1}  be an integer and $\xi$ an algebraic number of degree $d> n$.  Then
\M{
\omega_n(\xi)=\omegahat_n(\xi)= n.
}
\end{thm}

The spectrum of  the  exponent $\omega_n$ is $[n,+\infty]$. For $\omegahat_n$ it is not completely known.   From 
Theorem \ref{T:DavenportSchmidtCriterion} one deduces:

 \begin{thm}[Davenport and Schmidt] \label{T:CritereTdceDS}
For any real number $\xi$ which is not algebraic of degree $\le n$,
$$
\omegahat_n(\xi)\le 2n-1.
$$

\end{thm}
 


  
For the special case $n=2$ a sharper estimate holds: from Theorem \ref{Th:ArbourRoy} of B.~Arbour and D.~Roy one deduces
\begin{equation}\label{E:omegadeuxhat}
\omegahat_2 (\xi)\le \gamma+1
\end{equation}
(recall that  \m{\gamma} denotes the Golden Ratio \m{(1+\sqrt 5)/2}).

In \cite{MR0246822}, 
Davenport and Schmidt comment:  

\begin{quote}
 {\it
``We have no reason to think that the exponents in these theorems are best possible.''
}
\end{quote}

It was widely believed during a while that $\omegahat_2 (\xi)$ would turn out to be equal to $2$ for all $\xi\in\bR$ which are not rational nor quadratic irrationals. Indeed, otherwise, for $\kappa $ in the interval 
$2<\kappa<\omegahat_2 (\xi)$, the inequalities 
\begin{equation}\label{E:omegachapeau2}
0<|x_0+x_1\xi+x_2\xi^2|\le c N^{-\kappa}, \quad 
\max\{|x_0|,\; |x_1|,\; |x_2|\}\le N
\end{equation}
would have, for a suitable constant $c>0$ and for all sufficiently large $N$, a non--trivial solution in  integers $(x_0,x_1,x_2)\in\bZ^3$. However these inequalities define a convex body whose volume tends to zero as $N$ tends to infinity.  In such circumstances one does not expect a non--trivial solution to exist \footnote{Compare with the definition of {\it singular systems} in \S~7, Chap.~V of \cite{MR0087708}.}.  In general $\omegahat_2(\alpha,\beta)$ may be infinite  (Khintchine, 1926 --- see  \cite{MR0087708}) - however, here, we have the restriction $\beta=\alpha^2$.

Hence it came as a surprise  in 2003 when  D.~Roy   \cite{MR2003m:11103}
showed  
 that the estimate   (\ref{E:omegadeuxhat}) is optimal, by constructing examples of real (transcendental) numbers $\xi$ for which  $  \omegahat_2 (\xi)= \gamma+1=2.618\dots$

By means of a transference principle of Jarn\'{\i}k (Th.\ 2 of
\cite{JFM64.0145.01}), 
Th.\ 1 of \cite{MR2018957}  
can be reformulated as follows (see also Th.\ 1.5 of \cite{MR2076601}). 


\begin{thm}[Roy]\label{T:Roy}
There exists a real number $\xi$ which is neither rational nor a quadratic irrational and which has the property that, for a suitable constant $c>0$,  for all sufficiently large  integers  \m{N},  the inequalities
(\ref{E:omegachapeau2}) 
have a solution $(x_0,x_1,x_2)\in\bZ^2$ with $\kappa=\gamma+1$. 
  Any such number is transcendental over $\bQ$ and the set of such real numbers is countable. 
   
\end{thm}

In \cite{RoyDamienTwoExponents},
answering a question of Y. Bugeaud and M. Laurent \cite{MR2149403}, D.~Roy  shows that the exponents $  \omegahat_2 (\xi)$, where $\xi$ ranges through all real numbers which are not algebraic of degree $\le 2$, form a dense subset of the interval $[2, 1+\gamma]$.  

D.~Roy calls {\it extremal} a number which satisfies the conditions of Theorem \ref{T:Roy}; from the point of view of approximation by quadratic polynomials,  these numbers present a closest behaviour to quadratic real numbers. 

Here is the first example \cite{MR2003m:11103} of an extremal number $\xi$.
 Recall that {\it the Fibonacci word}
\M{
w=abaababaabaababaababaabaababaabaab\dots
}
is the fixed point of the morphism \m{a\mapsto ab},  \m{b\mapsto a}. It is the limit of the sequence of words starting with 
\m{f_1=b} and
\m{f_2=a}
and defined inductively by concatenation as
\m{
f_n=f_{n-1}f_{n-2}
}. 
Now let \m{A} and \m{B} be two distinct positive integers and
let \m{\xi\in(0,1)} be the real number whose continued fraction  expansion  is obtained from the Fibonacci word \m{w} by replacing the letters \m{a} and \m{b} by   \m{A} and \m{B}:
\M{
[0; A, B, A, A, B, A, B, A, A, B, A, A, B, A, B, A, A, 
\dots]
} 
Then $\xi$ is extremal. 

In \cite{MR2003m:11103,MR2018957,RoyDamienContinuedFraction},  D.~Roy  investigates the approximation properties of extremal numbers  by rational numbers, by quadratic numbers 
as well as by cubic integers. 

\begin{thm}[Roy]\label{T:RoyExtremal}
Let $\xi$ be an extremal number. There exist positive constants $c_1,\ldots,c_5$ with the following properties:
\\
(1) For any rational number $\alpha\in\bQ$ we have
$$
|\xi-\alpha|\ge c_1 H^{-2}(\log H)^{-c_2}
$$
with $H=\max\{2,\rmH(\alpha)\}$. 
\\
(2) For any algebraic number $\alpha$  of degree at most $2$ we have
$$
|\xi-\alpha|\ge c_3\rmH(\alpha)^{-2\gamma-2}
$$\\
(3) There exist infinitely many quadratic real numbers $\alpha$   with
$$
|\xi-\alpha|\le c_4\rmH(\alpha)^{-2\gamma-2}
$$\\
(4) For any algebraic integer $\alpha$  of degree at most $3$ we have
$$
|\xi-\alpha|\ge c_5\rmH(\alpha)^{-\gamma-2}.
$$
\end{thm}

Moreover, in \cite{MR2031862}, he shows that for some extremal numbers $\xi$,  property (4)  holds with the exponent $-\gamma-1$ in place of $-\gamma-2$.

In  \cite{MR2076601} D.~Roy  describes the method of Davenport and Schmidt and he gives a sketch of proof of his construction of extremal numbers. 

In
\cite{RoyDamienContinuedFraction} he gives a sufficient condition for an extremal number to have bounded quotients and constructs new examples of such numbers. 
     
The values of the different exponents for the extremal numbers which are associated with Sturmian words (including the Fibonacci word) have been obtained by Y.~Bugeaud and M.~Laurent \cite{MR2149403}. Furthermore, they show that the spectrum $\{\omegahat_2(\xi)\; ;\; \xi\in\bR\setminus\Qbar\}$ is not countable.   
See also their joint works \cite{BugeaudLaurentInhom,BugeaudLaurentExpDiophAppxDimTwo}.
Their method involves words with many palindromic prefixes. 
S.~Fischler in \cite{MR2110935,FischlerPalindromes} defines new exponents of approximation which allow him to obtain a characterization of the values of $\omegahat_2(\xi)$ obtained by these authors. 
    
In     \cite{AdamczewskiBugeaudMesures}, B.~Adamczewski and Y.~Bugeaud prove that for any extremal number $\xi$, there exists a constant $c=c(\xi)$ such that for any integer $n\ge 1$, 
$$
\omega_n(\xi)\le \exp\bigl\{c \bigl(\log (3n)\bigr)^2\bigl(\log\log(3n)\bigr)^2\bigr\}.
$$
 In particular an extremal number is either a $S$--number or a $T$--number in Mahler's classification.

Recent results on simultaneous approximation to a number and its square, on approximation to real numbers by quadratic integers and on quadratic approximation to numbers associated with {\it Sturmian words}  have been obtained by
{M.~Laurent, Y.~Bugeaud, S.~Fischler, D.~Roy\dots}

\subsection{Simultaneous rational approximation to powers of a real number
}
\label{SS:PASASN-SRARN}

     Let $\xi$ be a real number and \m{n} a positive integer. 

We consider first the simultaneous rational approximation of successive powers of $\xi$. We denote by  \m{\omega'_n(\xi)} the supremum of the real numbers $w$ for which there exist    {infinitely many} positive integers  \m{N} for which the system 
\begin{equation}\label{E:simultaneousappxpowers}
  0< \max_{1\le i\le n} |x_i-x_0\xi^i |\le N^{-w},
   \quad 
    \text{with}\quad
\max_{0\le i\le n} |x_i| \le N,
  \end{equation}
  has a solution in rational integers \m{x_0,x_1,\ldots,x_n}.
  
     An upper bound for  \m{ \omega'_n(\xi)}  yields a {\it simultaneous approximation measure} for 
  \m{\xi,\xi^2,\ldots,\xi^n}.

 Next   the uniform simultaneous approximation measure is the supremum   \m{\omegahat'_n(\xi)}  of the real numbers $w$ such that, for    any sufficiently large  integer  \m{N}, the same  system 
(\ref{E:simultaneousappxpowers})  has a solution in rational integers \m{x_0,x_1,\ldots,x_n}.

Notice that for \m{n=1},  \m{ \omega'_1(\xi) =\omega(\xi)} and \m{ \omegahat'_1(\xi)=\omegahat(\xi)}.

 According to Dirichlet's box principle, {\it for all $\xi$ and \m{n}, }
$$
\frac{1}{n}\le \omegahat'_n(\xi)\le 
\omega'_n(\xi).
$$

Khintchine's transference principle (see Th.\ B.5 in \cite{MR2136100} and Theorem \ref{T:KhinchineTransfer} below) yields relations between $\omega'_n$ and $\omega_n$. As remarked in Theorem 2.2 of \cite{MR2149403},  the same proof yields similar relations  between $\omegahat'_n$ and $\omegahat_n$. 

\begin{thm} 
Let $n$ be a positive integer and $\xi$ a real number which is not algebraic of degree $\le n$. Then
$$
\frac{1}{n}\le
\frac{\omega_n(\xi) }{(n-1)\omega_n(\xi)+n}\le 
\omega'_n(\xi)\le  
\frac{\omega_n(\xi)-n+1}{n}
$$
and
$$
\frac{1}{n}\le
\frac{\omegahat_n(\xi) }{(n-1)\omegahat_n(\xi)+n}\le 
\omegahat'_n(\xi)\le  
\frac{\omegahat_n(\xi)-n+1}{n}
\cdotp
$$

\end{thm}

The second set of inequalities follows from the inequalities (4) and
(5) of V.~Jarn\'{\i}k in Th.\ 3 of
\cite{JFM64.0145.01}, 
with conditional
refinements given by the inequalities (6) and (7) of the same
theorem.


 In particular, $\omega_n(\xi) =n$ if and only if $\omega'_n(\xi) =1/n$. 
Also, $\omegahat_n(\xi) =n$ if and only if $\omegahat'_n(\xi) =1/n$. 

\medskip
 
The spectrum of  $\omega'_n(\xi) $, where $\xi$ ranges over the set of real numbers which are not algebraic of degree $\le n$, is investigated by Y.~Bugeaud and M.~Laurent in \cite{BugeaudLaurentInhom}.   Only the case $n=2$ is completely solved.

It follows from Theorem \ref{T:omega_nPresquePartout} that {\it for almost all real numbers $\xi$, }
$$
\omega'_n(\xi) = \omegahat'_n(\xi) =  \frac{1}{n}
\quad\text{for all $n\ge 1$}.
$$
Moreover, a consequence of Schmidt's Theorem \ref{T:ConsequenceSubspace} is that {\it     for all $n\ge 1$   and for all algebraic real numbers $\xi$ of degree $d>n$, }
$$
\omega'_n(\xi) = \omegahat'_n(\xi) =  \frac{1}{n}
= \frac{1}{\omega_n(\xi)} 
\cdotp 
$$
 
Theorems 2a and 4a of the paper  \cite{MR0246822} by H.~Davenport and W.M.~Schmidt (1969)
imply that upper bounds for   \m{\omegahat'_n(\xi)} are valid for all  real numbers $\xi$ which are not algebraic of degree \m{\le n}.   
For instance,
$$
\omegahat'_1(\xi)=1,\quad \omegahat'_2(\xi)\le 1/\gamma=0.618 \dots, \quad \omegahat'_3(\xi)\le 1/2. 
$$
A slight  refinement was obtained by  M.~Laurent \cite{MR2015598} in 2003 (for the odd values of $n\ge 5$). 

     \begin{thm}[Davenport and Schmidt, Laurent]\label{T:DS2a}
 Let \m{\xi\in\bR\setminus\bQ} and \m{n\ge 2}.     
      Assume $\xi$ is not algebraic of degree \m{\le \lceil n/2\rceil}. Then   
$$
     \omegahat'_n(\xi)\le \lceil n/2\rceil^{-1}=\begin{cases}
      2/n & \text{if \m{n} is even,}\\
      2/(n+1) & \text{if \m{n} is odd.}\\
      \end{cases}
$$ 
\end{thm}

The definition of $\omegahat'_n$ with a supremum does not reflect the accuracy of the results in  \cite{MR0246822}; for instance, the upper bound $ \omegahat'_2(\xi)\le 1/\gamma$ is not as sharp as Theorem 1a of \cite{MR0246822} which is the following:

\begin{thm}[Davenport and Schmidt]\label{T:omegahat2}
Let $\xi$ be a real number which is not rational nor a quadratic irrational. There exists a constant $c>0$ such that, for arbitrarily large values of $N$, the inequalities 
$$
\max \{|x_1-x_0\xi |, |x_2-x_0\xi^2 | \}
\le c N^{-1/\gamma},
   \quad  
 |x_0|  \le N
  $$
have no solution $(x_0,x_1,x_2)\in\bZ^3$.

\end{thm}

Before restricting ourselves to the small values of $n$, we emphasise that there is a huge lack in our knowledge of  the spectrum of the set 
 \M{
 \bigl(\omega_n(\xi),\; \omegahat_n(\xi), \; \omega'_n(\xi),\; \omegahat'_n(\xi)\bigr)\in\bR^4,
 }
 where $\xi$ ranges over the set of real numbers which are not algebraic of degree \m{\le n}.

Consider  the special case \m{n=2} and the question of quadratic approximation.
 As pointed out by Y.~Bugeaud, a formula due to V.~Jarn\'\i k   (1938) 
(Theorem 1 of \cite{JFM64.0145.01};
see also Corollary A3 in \cite{RoyDamienTwoExponents} and \cite{LaurentMichel0703146}) relates $\omegahat_2$ and $\omegahat'_2$:
\begin{equation}\label{E:JarnikDim2}
\omegahat'_2(\xi)=1-\frac{1}{\omegahat_2(\xi)}\cdotp
\end{equation}
Therefore the properties of $\omegahat_2$ which we considered in \S~\ref{SS:PASASN-PACN} can be translated into properties of $\omegahat'_2$. 
For instance,
$\omegahat'_2(\xi)=1/2$ if and only if $\omegahat_2(\xi)=2$, and this holds for almost all $\xi\in\bR$ (see Theorem \ref{T:omega_nPresquePartout}) and for all algebraic real numbers $\xi$ of degree \m{\ge 3} (see Theorem \ref{T:ConsequenceSubspace}).
 If $\xi\in\bR$ is neither rational nor a quadratic irrational,  
 Davenport and Schmidt  have proved
 \begin{equation}\label{E:DSomegadeuxhat}
 \omegahat'_2(\xi)\le 1/\gamma=0.618\dots  
\end{equation}
 The extremal numbers of D.~Roy in Theorem \ref{T:Roy} satisfy
$\omegahat'_2 (\xi)=1/\gamma$. More precisely, they are exactly the numbers $\xi\in\bR$ which are not rational nor quadratic irrationals and satisfy the following property:   {\it there exists a constant $c>0$ such that, for any sufficiently large number $N$,   the inequalities 
$$
\max \{|x_1-x_0\xi |, |x_2-x_0\xi^2 | \}
\le c N^{-1/\gamma},
   \quad  
  0<\max \{|x_0| , \; |x_1| ,\; |x_2| \}\le N,
$$
  have a solution in rational integers $x_0,x_1, x_2$.} (This was the original definition). 

In \cite{RoyDamienTwoExponents}, using Jarn\'\i k's formula (\ref{E:JarnikDim2}),  D.~Roy shows that the set of $(\omegahat'_2(\xi),\omegahat'_2(\xi))\in\bR^2$, where  $\xi$ ranges over the set of real numbers which are not algebraic of degree $\le 2$, is dense in the piece of curve
$$
\{(1-t^{-1},t)\; ; \; 2\le t\le \gamma+1\}.
$$

We conclude with  the  case \m{n=3} and the question of cubic approximation. When $\xi\in\bR$ is not algebraic of degree $\le 3$, the estimate for $\omegahat'_3(\xi)$  by  Davenport and Schmidt  \cite{MR0246822}  is 
\M{
\frac{1}{3}\le \omegahat'_3(\xi)\le  \frac{1}{2}\cdotp
}
As we have seen, the lower bound is optimal (equality holds for almost all numbers and all algebraic numbers of degree $>n$). The upper bound has  been improved by 
D.~Roy in 
 \cite{RoyDamienContinuedFraction} 
$$
\omegahat'_3(\xi)\le  \frac{1}{2}
(2\gamma+1-\sqrt{4\gamma^2+1})=0.4245\dots 
$$

\subsection{Algebraic approximation to a single number
}
\label{SS:PASASN-AACN}

  Let $\xi$ be a real number and \m{n} a positive integer. 

   Denote by \m{\omega^*_n(\xi)} the supremum of the real numbers $w$ for which  there exist  infinitely many positive integers  \m{N}  with the following property: {\it  there exists an algebraic number $\alpha$ of degree $\le n$ and height $\le N$ satisfying\footnote{The occurrence of $-1$ in the exponent of the right hand side of (\ref{E:moinsun}) is already plain for degree $1$ polynomials, comparing $|\alpha-p/q|$ and $|q\alpha-p|$.}
   }
\begin{equation}\label{E:moinsun}
0<   |\xi-\alpha|\le N^{-w-1}.
 \end{equation}    
 
An upper bound for   \m{ \omega^*_n(\xi)}   is a {\it measure of algebraic approximation} for 
  \m{ \xi}. 
   These numbers arise in Koksma's classification of complex numbers  (Chap.~III  \S~3 of \cite{MR19:252f}   and \S~3.3 of \cite{MR2136100}).
   

  Next, denote by \m{\omegahat^*_n(\xi)} the supremum of the real numbers $w$ such that, 
   {\it for any sufficiently large  integer  \m{N},  there exists an algebraic number $\alpha$ of degree $\le n$ and height $\le N$ satisfying}
      \M{
   |\xi-\alpha|\le \rmH(\alpha)^{-1}N^{-w}.
  }

   
  An upper bound for   \m{\omegahat^*_n(\xi)} yields a {\it uniform measure of algebraic approximation} for 
  \m{ \xi}. 
  
  From Schmidt's Subspace Theorem one deduces, for a real algebraic number $\xi$ of degree $d$ and for $n\ge 1$, 
 $$
 \omegahat^*_n(\xi)= \omega^*_n(\xi)=\min\{n, d-1\}.
 $$
 See \cite{MR2136100} Th.\ 2.9 and 2.11. 
  
       That there are relations between $\omega_n$ and $\omega^*_n$  (and, for the same reason, between $\omegahat_n$ and $\omegahat^*_n$) can be expected from Lemmas  \ref{L:EasyPart} and \ref{L:XimoinsGamma}. Indeed, a lot of information on these numbers has been devised in  order to compare the classifications of Mahler and Koksma. The estimate 
$$
\omega_n(\xi)\ge \omega_n^*(\xi),
$$ 
which follows from Lemma \ref{L:EasyPart}, was known by Koksma
(see also Wirsing's paper \cite {MR0142510}). In the reversed direction, the inequalities 
\begin{equation}\label{E:WirsingA}
\omega_n^*(\xi) \ge \omega_n(\xi)-n+1,\quad   \omega_n^*(\xi) \ge \frac{\omega_n(\xi)+1}{2}
\end{equation}
and
\begin{equation}\label{E:Wirsing}
       \omega^*_n(\xi)\ge \frac{\omega_n(\xi)}{\omega_n(\xi)-n+1}
\end{equation}
have been obtained by  E.~Wirsing in 1960
         \cite {MR0142510} (see \S~3.4 of \cite{MR2136100}). 
         
A consequence  is that for a real number $\xi$ which is not algebraic of degree $\le n$, if
$\omega_n(\xi) =n$  then  $\omega^*_n(\xi)=n$.

The  inequality (\ref{E:Wirsing})  of Wirsing has been refined in Theorem 2.1 of  \cite{MR2149403} as follows.

\begin{thm}[Bugeaud and Laurent]\label{T:BugeaudLaurentFourier}
Let $n$ be a positive integer and $\xi$ a real number which is not algebraic of degree $\le n$. Then
$$
 \omegahat^*_n(\xi)\ge \frac{\omega_n(\xi)}{\omega_n(\xi)-n+1}
\quad\text{
and
} \quad
 \omega^*_n(\xi) \ge \frac{\omegahat_n(\xi)}{\omegahat_n(\xi)-n+1}\cdotp
$$

\end{thm}

         A number of recent papers are devoted to this topic, including  the survey given in the first part of \cite{MR2149403}
as well as  Bugeaud's  papers
\cite{MR2001g:11113,MR2007546,MR1981928,BugeaudMahlerClassification2,MR2107949}
where further references can be found. 

We quote Proposition 2.1 of  \cite{MR2149403} which gives connections between the six exponents 
$\omega_n$, $\omegahat_n$, $\omega'_n$, $\omegahat'_n$, $\omega^*_n$, $\omegahat^*_n$.

\begin{prop}
Let $n$ be a positive integer and $\xi$ a real number which is not algebraic of degree $\le n$. Then
$$
\frac{1}{n}\le \omegahat'_n(\xi)\le \min\{1,\omega'_n(\xi)\}
$$
and
$$
1\le \omegahat^*_n(\xi)\le \min\bigl\{
\omega_n^*(\xi),\; \omegahat_n(\xi)\bigr\}
\le \max\bigl\{
\omega_n^*(\xi),\; \omegahat_n(\xi)\bigr\}
\le \omega_n(\xi).
$$

\end{prop}

A further relation connecting 
$\omega^*_n$ and $\omegahat'_n$ has been discovered by H.~Davenport and W.M.~Schmidt in 1969  \cite{MR0246822}. We discuss their contribution in \S~\ref{SS:PASASN-AAI}. For our immediate concern here we only quote the following result:

\begin{thm}
\label{T:DSdualite}
Let $n$ be a positive integer  and  $\xi$ a real number which is not  algebraic of degree \m{\le n}, 
Then
$$
\omega^*_n(\xi)\omegahat'_n(\xi)\ge 1.
$$
\end{thm}

The spectral question  for $\omega^*_n$ is one of the main challenges in this domain. 
Wirsing's conjecture  states that  for any integer $n\ge 1$ and any real number $\xi$ which is not algebraic of degree $\le n$,  we have $\omega^*_n(\xi)\ge n$. In other terms:

   \begin{conj}[Wirsing]\label{C:Wirsing}
For any \m{\epsilon>0} there is a constant \m{c(\xi,n,\epsilon) >0} for which there are infinitely many algebraic numbers \m{\alpha} of degree \m{\le n} with
\M{
|\xi-\alpha|\le c(\xi,n,\epsilon) \rmH(\alpha)^{-n-1+\epsilon}.
}
\end{conj}

In 1960, E.~Wirsing   \cite {MR0142510} proved that for any real number which is not algebraic of degree \m{\le n},  the lower bound 
 $\omega^*_n(\xi)\ge (n+1)/2$ holds: it suffices to combine (\ref{E:WirsingA})  with the lower bound $\omega_n(\xi)\ge n$ from (\ref{E:MinorationOmega-n}) (see \cite{0421.10019} Chap.~VIII Th.\ 3B). More precisely, he proved  that   for such a $\xi\in\bR$ 
  {\it there is a constant $ c(\xi,n)>0 $ for which  there exist infinitely many algebraic numbers \m{\alpha} of degree \m{\le n} with
\M{
|\xi-\alpha|\le c(\xi,n) \rmH(\alpha)^{-(n+3)/2}.
}
}
The special case $n=2$ of this estimate was improved in 1967 when
H.~Davenport and W.M.~Schmidt   \cite{MR0219476}  replaced 
\m{(n+3)/2=5/2} by \m{3}.
This is optimal for the approximation to a real number by quadratic algebraic numbers. This is the only case where   Wirsing's Conjecture is solved. 
More recent estimates are due to V.I.~Bernik and K.~Tishchenko 
\cite{MR1282119,
MR1827705,
MR1826201,
MR1843249,
MR1850023,
TishchenkoJNT2007}. This question is studied by Y.~Bugeaud in his book   \cite{MR2136100} (\S~3.4) where he proposes the following {\it Main Problem}:

\begin{conj}[Bugeaud]\label{C:BugeaudMainPb}
Let $(w_n)_{n\ge 1}$ and $(w^*_n)_{n\ge 1}$ be two non--decreasing sequences in $[1,+\infty]$ for which 
$$
n\le w^*_n\le w_n\le w^*_n+n-1\quad\text{for any $n\ge 1$}.
$$
Then there exists a transcendental real number $\xi$ for which 
$$
\omega_n(\xi)=w_n \quad\text{and}\quad
\omega^*_n(\xi)=w^*_n\quad\text{for any $n\ge 1$}.
$$ 
\end{conj}  

A summary of known results on this problem is given in \S~7.8 of  \cite{MR2136100}. 
   
The spectrum 
$$
\bigl\{ \omega_n(\xi)- \omega^*_n(\xi) \; ; \;
\text{ $\xi\in\bR$ not algebraic of degree $\le n$}\bigr\}\subset [0,n-1] 
$$
of $\omega_n-\omega_n^*$ for $n\ge 2$ was studied by R.C.~Baker in 1976 who showed that it contains $[0,1-(1/n)]$. 
This has been improved by  Y.~Bugeaud 
  in \cite{MR2007546}:  it contains the interval $[0,n/4]$.
   

Most results concerning $\omega^*_n(\xi)$ and  $\omegahat^*_n(\xi)$ for $\xi\in\bR$ have extensions to complex numbers, only the numerical estimates are slightly different. However see \cite{BugeaudEvertse2007}.

\subsection{Approximation by algebraic integers
}
\label{SS:PASASN-AAI}

An innovative and  powerful  approach was initiated in the seminal paper \cite{MR0246822}  by H.~Davenport and W.M.~Schmidt (1969). It 
 rests on the  
 transference principle arising from the geometry of numbers and Mahler's theory of {\it  polar convex bodies} and allows   to deal with approximation by algebraic integers of  bounded degree. The next statement includes  a refinement by  Y.~Bugeaud and O.~Teulié (2000) \cite{MR1760090} who observed that one may treat approximations by algebraic integers of given degree; the sharpest results in this direction are due to M.~Laurent \cite{MR2015598}.

From the estimate $\omega_n^*(\xi)\omegahat'_n(\xi)\ge 1$ in Theorem~\ref{T:DSdualite} one deduces the following statement. Let $n$ be a positive integer  and let $\xi$ be a real number which is not   algebraic of degree \m{\le n}.  Let $\lambda$ satisfy  \m{\omegahat'_n(\xi) < \lambda}. Then
for  $\kappa=(1/\lambda)+1$,   there is a constant \m{ c(n,\xi,\kappa)>0} such that the equation
\begin{equation}\label{E:PolarBodies}
|\xi-\alpha| \le c(n,\xi,\kappa) \rmH(\alpha)^{-\kappa} 
\end{equation}
has infinitely many  solutions in
algebraic numbers  \m{\alpha} of degree \m{n}. 
In this statement one may replace ``algebraic numbers  \m{\alpha} of degree \m{n}''  
by 
``algebraic integers  \m{\alpha} of degree \m{n+1}'' 
and also by 
``algebraic units  \m{\alpha} of degree \m{n+2}''. 
   
\begin{prop}\label{P:Polar}
  Let $\kappa>1$ be a   real number,  \m{n} be a positive integer and $\xi$ be a real number which  is not algebraic of degree \m{\le n}.    
 Assume \m{\omegahat'_n(\xi)<  1/(\kappa-1) }. Then there exists a constant  \m{ c(n,\xi,\kappa)>0} such that there are infinitely many 
algebraic integers  \m{\alpha} of degree \m{n+1} satisfying (\ref{E:PolarBodies})
and there  are infinitely many 
algebraic units  \m{\alpha} of degree \m{n+2} satisfying (\ref{E:PolarBodies}).

\end{prop}

Suitable values for $\kappa$ are deduced from Theorem \ref{T:DS2a} and estimate  (\ref{E:omegadeuxhat}). 
For instance, from Theorem~\ref{T:DSdualite} and the estimate $ \omegahat'_2(\xi)\le 1/\gamma$ of Davenport and Schmidt in (\ref{E:DSomegadeuxhat}) one deduces 
$ \omega^*_2(\xi)\ge \gamma$. Hence for any $\kappa<1+\gamma$ the 
 the assumptions of Proposition \ref{P:Polar} are satisfied. More precisely, the duality (or transference) arguments used by Davenport and Schmidt to prove  Theorem~\ref{T:DSdualite} together with their Theorem \ref{T:omegahat2} enabled them to deduce the next statement (\cite{MR0246822}, Th.\ 1).

\begin{thm}[Davenport and Schmidt]\label{T:DSomega2etoile}
Let $\xi\in\bR$ be a real number which is neither rational nor a quadratic irrational. Then there is a constant $c>0$ with the following property: there are infinitely many algebraic integers $\alpha$ 
of degree at most $3$ which satisfy
\begin{equation}\label{E:DSomega2etoile}
0<|\xi-\alpha|\le c \rmH(\alpha)^{-\gamma-1}.
\end{equation}
\end{thm}

Lemma \ref{L:EasyPart} shows that under the same assumptions, for another constant $c>0$  there are infinitely many monic polynomials $P\in\bZ[X]$ of degree at most $3$ satisfying
\begin{equation}\label{EasyPart}
|P(\xi)|\le c \rmH(P)^{-\gamma}.
\end{equation}
Estimates (\ref{E:DSomega2etoile}) and (\ref{EasyPart}) are optimal for certain classes of {\it extremal} numbers \cite{MR2031862}.    
Approximation of extremal numbers by cubic integers are studied by D.~Roy in  \cite{MR2031862,
MR2018957}.
Further papers dealing with approximation by algebraic integers include
\cite{MR1760090,MR1850023,MR2041771,1131.11046}. 

 Another development of the general and powerful method of Davenport and Schmidt deals with the question of approximating simultaneously several numbers by conjugate algebraic numbers: this is done in \cite{MR2041771} and refined in \cite{MR2175097} by D.~Roy. Also in \cite{MR2175097} D.~Roy gives variants of Gel'fond's Transcendence Criterion involving not only   a single number $\xi$ but  sets $\{\gamma+\xi_1,\ldots,\gamma+\xi_m\}$ or  $\{\gamma\xi_1,\ldots,\gamma\xi_m\}$. 
 In two recent manuscripts \cite{RoyDamienSmallValueAdditive,RoyDamienSmallValueMultplicative}, D.~Roy produces new criteria for the additive and for the multiplicative groups.
 
A different  application of  transference theorems is to link 
inhomogeneous Diophantine approximation problems   with homogeneous ones     \cite{BugeaudLaurentInhom}.

\subsection{Overview of metrical results for polynomials}\label{SS:OverviewMetricalPolynomials}

Here we give a brief account of some significant
results that have
produced new ideas and generalisations, as well as some interesting
problems and conjectures.
We begin with the probabilistic theory (that is, Lebesgue measure
statements) and continue with the more delicate Hausdorff
measure/dimension results. Results for multivariable polynomials, in
particular, the recent proof of a conjecture of Nesterenko on the
measure of algebraic independence of almost all real $m$-tuples,
will be sketched in \S\,\ref{SS:FurtherMetricalResults}, as will metrical results on
simultaneous approximation. Note that many of the results suggested
here have been established in the far more general situation of
Diophantine approximation on manifolds. However, for simplicity, we
will only explain this Diophantine approximation for the case of
integral polynomials.

\bigskip

Mahler's problem \cite{Mahler-1932b}, which arose from his
classification of real (and complex) numbers, remains a major
influence over the metrical theory of Diophantine approximation. As
mentioned in \S\ref{SS:PASASN-PACN}, 
the problem has been settled by Sprindzuk in
1965.  Answering a question posed by A.~Baker in
\cite{Baker-1966-On-a-theorem-of-Sprindzuk},
Bernik~\cite{Bernik-1989} established a generalisation of Mahler's
problem akin to Khintchine's one-dimensional convergence result in Theorem \ref{T:Khinchin},
involving the critical sum
\begin{equation}\label{e1}
    \sum_{h=1}^\infty \Psi(h)
\end{equation}
of values of the function $\Psi:\N\to\R^+$ that defines the error of
approximation.
\begin{thm}[Bernik, 1989]\label{bernik1}
Given a monotonic $\Psi$ such that the critical sum~\eqref{e1}
converges, for almost all $\xi\in\R$ the inequality
\begin{equation}\label{e2}
    |P(\xi)|<H(P)^{-n+1}\Psi(H(P))
\end{equation}
has only finitely many solutions in $P\in\Z[x]$ with $\deg P\le n$.
\end{thm}

In the case $n=1$, inequality~\eqref{e2} reduces to rational
approximations of real numbers and is covered by Khintchine's
theorem   \ref{T:Khinchin} \cite{Khintchine-1924}. Khintchine's theorem  \ref{T:Khinchin} also covers the
solubility of~\eqref{e2} in case when $n=1$ and~\eqref{e1} diverges.
For arbitrary $n$ the complementary divergence case of
Theorem~\ref{bernik1} has been established by Beresnevich, who has
shown in \cite{Beresnevich-99:MR1709049} that if\/ $\eqref{e1}$
diverges then for almost all real $\xi$ inequality $\eqref{e2}$ has
infinitely many solutions  $P\in\Z[x]$ with $\deg P=n$. In fact the
latter statement follows from the following analogue of Khintchine's
theorem  \ref{T:Khinchin} for approximation by algebraic numbers, also established in
\cite{Beresnevich-99:MR1709049}.

\begin{thm}[Beresnevich, 1999]\label{beresnevich1} Let $n\in\N$,
$\Psi:\N\to\R^+$ be a monotonic error
function and $\cA_n(\Psi)$ be the set of real $\xi$ such that
\begin{equation}\label{e3}
    |\xi-\alpha|<H(\alpha)^{-n}\Psi(H(\alpha))
\end{equation}
has infinitely many solutions in real algebraic numbers of degree
$\deg\alpha=n$. Then $\cA_n(\Psi)$ has full Lebesgue measure if the
sum $(\ref{e1})$ diverges and zero Lebesgue measure otherwise.
\end{thm}

\noindent Bugeaud \cite{Bugeaud-?2} has proved an analogue of
Theorem~\ref{beresnevich1} for approximation by algebraic integers:

\begin{thm}[Bugeaud, 2002]\label{bugeaud1} Let $n\in\N$, $n\ge 2$,
 $\Psi:\N\to\R^+$ be a monotonic error
function and $\mathcal{I}_n(\Psi)$ be the set of real $\xi$ such
that
\begin{equation}\label{e4}
    |\xi-\alpha|<H(\alpha)^{-n+1}\Psi(H(\alpha))
\end{equation}
has infinitely many solutions in real algebraic integers of degree
$\deg\alpha=n$. Then $\mathcal{I}_n(\Psi)$ has full Lebesgue measure
if the sum $(\ref{e1})$ diverges and zero Lebesgue measure
otherwise.
\end{thm}

No analogue of Theorem~\ref{bernik1} is known for the monic polynomial case
--- however see the final section of
\cite{Beresnevich-Velani-08-Inhom}.

\bigskip

Unlike Theorem~\ref{bernik1}, the convergence parts of
Theorems~\ref{beresnevich1} and \ref{bugeaud1} are rather trivial
consequences of the Borel-Cantelli Lemma, which also implies that
the monotonicity condition is unnecessary in the case of
convergence. The intriguing question now arises whether the
monotonicity condition in Theorem~\ref{bernik1} and in the
divergence part of Theorems~\ref{beresnevich1} and \ref{bugeaud1}
can be dropped. Beresnevich \cite{Beresnevich-05:MR2110504} has
recently shown that the monotonicity condition on $\Psi$ can indeed
be safely removed from Theorem~\ref{bernik1}. Regarding
Theorems~\ref{beresnevich1} and \ref{bugeaud1} removing the
monotonicity condition is a fully open problem --~\cite{Beresnevich-05:MR2110504}. In fact, in dimension $n=1$
removing the monotonicity from Theorem~\ref{beresnevich1} falls
within the Duffin and Schaeffer problem
\cite{Duffin-Schaeffer-41:MR0004859}. Note, however, that the higher
dimensional Duffin-Schaeffer problem has been settled in the affirmative
\cite{Pollington-Vaughan-1990}.

It is interesting to compare Theorems~\ref{beresnevich1} and
\ref{bugeaud1} with their global counterparts. In the case of
approximation by algebraic numbers of degree $\le n$, the
appropriate statement is known as the Wirsing conjecture
\ref{C:Wirsing}
(see \S\,2.5). The latter has been verified for $n=2$ by Davenport
and Schmidt but is open in higher dimensions.
Theorem~\ref{beresnevich1} implies that the statement of the
Wirsing conjecture
\ref{C:Wirsing}
 holds for almost all real $\xi$ -- the
actual conjecture states that it is true at least for all
transcendental $\xi$. In the case of approximation by algebraic
integers of degree $\le n$, Roy has shown that the statement
analogous to Wirsing's conjecture is false \cite{MR2031862}.
However, Theorem~\ref{bugeaud1} implies that the statement holds for
almost all real $\xi$. In line with the recent `metrical' progress
on Littlewood's conjecture by Einsiedler, Katok and Lindenstrauss
\cite{Einsiedler-Katok-Lindenstrauss-06:MR2247967}, it would be
interesting to find out whether the set of possible exceptions to
the Wirsing-Schmidt conjecture is of Hausdorff dimension zero.  A
similar question can also be asked about approximation by algebraic
integers; this would shed light on the size of the set of
exceptions, shown to be non--empty by Roy.

\bigskip

A.~Baker \cite{Baker-75:MR0422171} suggested a strengthening of
Mahler's problem in which the height
$H(P)=\max\{|a_n|,\dots,|a_0|\}$ of polynomial
$P(x)=a_nx^n+\cdots+a_1x+a_0$ is replaced by
$$
H^\times(P)=\prod_{i=1}^n\max\{1,|a_i|\}^{1/n}.
$$
The corresponding statement has been established by Kleinbock and
Margulis \cite{Kleinbock-Margulis-98:MR1652916} in a more general
context of Diophantine approximation on manifolds. Specialising their
result to polynomials gives the following

\begin{thm}[Kleinbock \& Margulis, 1998]\label{km98}
Let $\ve>0$. Then for almost all $\xi\in\R$ the inequality
\begin{equation}\label{e5}
    |P(\xi)|<H^\times(P)^{-n-\ve}
\end{equation}
has only finitely many solutions in $P\in\Z[x]$ with $\deg P\le n$.
\end{thm}

A multiplicative analogue of Theorem~\ref{bernik1} with $H(P)$
replaced by $H^\times(P)$ has been obtained by Bernik, Kleinbock and
Margulis \cite{Bernik-Kleinbock-Margulis-01:MR1829381} (also within
the framework of manifolds). Note that in their theorem the
convergence of (\ref{e1}) must be replaced by the stronger condition
that $\sum_{h=1}^\infty \Psi(h)(\log h)^{n-1}<\infty$. This condition
is believed to be optimal but it is not known if the multiplicative
analogue of Theorem~\ref{beresnevich1}, when $H(P)$ is replaced by
$H^\times(P)$, holds. In \cite{Beresnevich-Velani-08-Inhom},
Beresnevich and Velani have proved an inhomogeneous version of the
theorem of Kleinbock and Margulis; and, in particular, an
inhomogeneous version of Theorem~\ref{km98}.

\bigskip

With \cite{BakerSchmidt-1970}, A.~Baker and W.M.~Schmidt pioneered
the use of Hausdorff dimension in the context of approximation of
real numbers by algebraic numbers with a natural generalisation of
the Jarn\'\i k--Besicovitch theorem:

\begin{thm}[Baker \& Schmidt, 1970]\label{BakerSchmidt}
Let $w\ge n$. Then the set of $\xi\in\R$ for which
\begin{equation}
  \label{eq:ineq}
  |\xi-\alpha|<H(\alpha)^{-w}
\end{equation}
holds for infinitely many algebraic numbers $\alpha$ with
$\deg\alpha\le n$ has Hausdorff dimension  \makebox{$(n+1)/(w+1)$}.
\end{thm}

In particular, Theorem~\ref{BakerSchmidt} implies that the set
\begin{equation}\label{e7}
   A(w)=\Big\{\xi\in\R:|P(\xi)|<H(P)^{-w}
\text{ for infinitely many }P\in\Z[x],\ \deg P\le n\Big\}
\end{equation}
has Hausdorff dimension at least $(n+1)/(w+1)$. Baker and Schmidt
conjectured that this lower bound is sharp, and this was established by Bernik in
\cite{Bernik-1983a}:

\begin{thm}[Bernik, 1983]\label{bernik2}
Let $w\ge n$. Then $\dim A(w)=\dfrac{n+1}{w+1}$.
\end{thm}

This theorem has an important consequence for the
spectrum of Diophantine exponents already discussed (see Chap.~5 of \cite{MR2136100}). 
Bugeaud
\cite{Bugeaud-?2} has obtained an analogue of
Theorem~\ref{BakerSchmidt} in the case of algebraic integers.
However, obtaining an analogue of Theorem~\ref{bernik2} for the case
of algebraic integers is as yet an open problem.

Recently, Beresnevich, Dickinson and Velani have established a
sharp Hausdorff measure version of Theorem~\ref{BakerSchmidt}, akin
to a classical result of Jarn\'\i k. In order to avoid introducing
various related technicalities, we refer the reader to
\cite[\S\,12.2]{Beresnevich-Dickinson-Velani-06:MR2184760}. Their
result implies the corresponding divergent statement for the
Hausdorff measure of the set of $\xi\in\R$ such that (\ref{e2})
holds infinitely often. Obtaining the corresponding convergent
statement represents yet another open problem.

There are various generalisations of the above results
to the case of complex and $p$-adic numbers and more generally to
the case of $S$-arithmetic (for instance by D.~Kleinbock and G.~Tomanov in \cite{1135.11037}).


\section{Simultaneous Diophantine approximation in higher dimensions}
\label{S:SDAHD}

In \S~\ref{SS:PASASN-PACN},  we considered polynomial approximation to a complex number $\xi$, which is the study of \m{|P(\xi)|}  for \m{P\in \bZ[X]}. As we have seen, negative results on the existence of polynomial approximations lead to {\it transcendence measures}. A more general situation is to fix several complex numbers \m{x_1,\ldots,x_m}
and to study the smallness of polynomials in these numbers -- negative results provide {\it measures of algebraic independence} to \m{x_1,\ldots,x_m}. 

This is a again a special case, where
 \m{\xi_i= x_1^{a_1}\cdots  x_m^{a_m}},  
of the study of linear combinations in \m{ \xi_1,\, \dots, \, \xi_n},  where $ \xi_1,\, \dots, \, \xi_n$ are given complex numbers. Now, negative results are {\it measures of linear independence
} to  $ \xi_1,\, \dots, \, \xi_n$. 

 There are still more general situations which we are not going to consider thoroughly but which are worth mentioning, namely the study of 
 {\it simultaneous approximation of dependent quantities} and {\it  approximation on a manifold} (see for instance
\cite{MR1727177}). 
   
   We start with the question of algebraic independence (\S~\ref{SS:SDAHD-CAI}) in connection with extensions to higher dimensions of Gel'fond's Criterion \ref{T:GTC}. Next (\S\S~\ref{SS:SDAHD-FourExponents}  and \ref{SS:SDAHD-FE}) we discuss a recent work by M.~Laurent 
   \cite{LaurentMichel0703146}, who introduces further coefficients for the study of simultaneous approximation. The special case of two numbers (\S~\ref{SS:D2})  is best understood so far. 
   
There is a very recent common generalisation of the question of Diophantine approximation to a point in $\bR^n$ which is considered in \S~\ref{SS:SDAHD-FE} on the one hand, and of the question of approximation to a real number by algebraic numbers of bounded degree considered in \S~\ref{SS:PASASN-AACN} on the other hand. It consists in the investigation of the approximation to a point in $\bR^n$ by algebraic hypersurfaces, or more generally algebraic varieties defined over the rationals. This topic has been recently investigated  by W.M.~Schmidt in   \cite{SchmidtTransactions2007}  and  \cite{SchmidtFunctionesApproximatio2006}.

\subsection{Criteria for algebraic independence}
\label{SS:SDAHD-CAI}

In    \S~\ref{SS:PASASN-GTC} we quoted  Gel'fond's algebraic independence results  of two numbers of the form $\alpha^{\beta^i}$ ($1\le i\le d-1$). His method has been extended after the work of several mathematicians including A.O.~Gel'fond, A.A.~Smelev, W.D.~Brownawell, G.V.~Chudnovskii, P.~Philippon, Yu.V.~Nesterenko, G.~Diaz  (see
\cite{MR22:2598}, 
 \cite{MR35:5397}, 
\cite{MR45:3333}, 
\cite{MR85i:11061,MR87e:11086}, 
 \cite{MR88h:11048,MR1837830}, 
 \cite{MR99a:11088b} Chap.~6 
 and \cite{MR1837822}). 
 So far the best known result, due to G.~Diaz \cite{MR978097}, proves ``half'' of what is expected. 

\begin{thm}\label{T:ai}
Let  \m{\beta} be an algebraic number of degree \m{d\ge 2} and $\alpha $ a non--zero algebraic number. Moreover, let $\log\alpha$ be any non--zero logarithm of $\alpha$.  Write $\alpha^z$ in place of $\exp(z\log\alpha)$. Then  among the numbers
 \M{
\alpha^{\beta}, \; \alpha^{\beta^2}, \; \dots \, , \alpha^{\beta^{d-1}}, 
}
at least \m{\lceil(d+1)/2\rceil} are algebraically independent.
\end{thm}

   \medskip
In order to prove such a result, as pointed out by S.~Lang in \cite{MR33:1286}, it would have been sufficient to replace 
 the Transcendence Criterion Theorem ~\ref{T:GTC}   by a criterion for algebraic independence. 
 However, an example, going back to A.Ya.~Khintchine in 1926  \cite{JFM52.0183.01} and quoted in J.W.S.~Cassels's book 
(\cite{MR0087708} Chap.~V, Th.\ 14; see also the appendix of \cite{MR88h:11048}  and appendix A of \cite{RoyDamienSmallValueAdditive}), shows that in higher dimensions, some extra hypothesis cannot be avoided (and this is a source of difficulty in the proof of Theorem \ref{T:ai}).
After the work of W.D.~Brownawell and G.V.~{\v{C}}udnovs{\cprime}ki\u{\i}, such criteria were proved by  P.~Philippon \cite{MR88h:11048,MR1837830}, Yu.V.~Nesterenko
   \cite{MR85i:11061,MR87e:11086}, M.~Ably, C.~Jadot (further references are given in \cite{MR99a:11088b} and 
    \cite{MR1756786} \S~15.5).  Reference  \cite{MR1837822} is an introduction to algebraic independence theory which 
 includes a chapter on multihomogeneous elimination 
by   G.~Rémond  \cite{MR1837827} and a discussion of criteria for algebraic independence  by P.~Philippon \cite{MR1837830}.  

Further progress has been made   by M.~Laurent and D.~Roy in 1999 who produced  criteria with multiplicities \cite{MR1475689,MR2000b:11088} (see also \cite{MR2002d:11091})  and  considered questions of approximation by algebraic sets. Moreover, in 
   \cite{MR2000b:11088} they investigate the approximation properties, by algebraic numbers of bounded degree and height, of a $m$-tuple which generates a field of transcendence degree $1$ -- this means that the corresponding point in  $\bC^m$ belongs to an affine curve defined over  $\bQ$.  For $m=1$ they proved in
   \cite{MR1475689} the
   existence of approximation; this has been improved by G.~Diaz in \cite{MR1451234}. Further  contributions are due to P.~Philippon (see for instance \cite{MR1752252}). 
   
A very special case of the investigation of Laurent and Roy  is a result related to Wirsing's lower bound for $\omega_n^*$ (see \S~\ref{SS:PASASN-AACN}), with a weak numerical constant, but with a lower bound for the degree of the approximation. Their result (Corollary 1 of \S~2 of \cite{MR2000b:11088})
has been improved by 
 Y.~Bugeaud and O.~Teulié (Corollary 5 of \cite{MR1760090}) who prove that the approximations $\alpha$ can be required to be algebraic numbers of exact degree $n$ or algebraic    integers  of exact degree $n+1$.

\begin{thm}[Bugeaud and Teulié]\label{T:BugeaudTeulie}
Let 
$\epsilon>0$ be a real number,
$n\ge 2$ be an integer and $\xi$  a real  number which is not algebraic of degree $n$. Then 
the inequality 
$$
|\xi-\alpha|\le  \rmH(\alpha)^{-((n+3)/2)+\epsilon}
$$
has infinitely many solutions in algebraic integers  $\alpha$ of degree $n$.

\end{thm}   

The $\epsilon$ in the exponent can be removed by introducing a constant factor. 
Further, for almost all $\alpha\in\bR$,  the result holds with the exponent  replaced by $n$, as shown by Y.~Bugeaud in \cite{1020.11049}.

There are close connections between questions of algebraic independence and simultaneous approximation of numbers. We shall not discuss this subject thoroughly here; it would deserve another survey. We just quote a few recent papers. 

 Applications of Diophantine approximation questions to transcendental number theory are considered by  P.~Philippon in 
   \cite{MR1453214}. 
M.~Laurent  \cite{0933.11039}  gives heuristic motivations in any transcendence degree.  
Conjecture 15.31 of \cite{MR1756786} 
 on simultaneous approximation of complex numbers suggests a path towards results on large transcendence degree.

 In \cite{MR1816056} D.~Roy shows some limitations of the conjectures of algebraic approximation by constructing points in $\bC^m$ which do not have good algebraic approximations of bounded degree and height, when the bounds on the degree and height are taken from specific sequences. The coordinates of these points are Liouville numbers.

\subsection{Four exponents: asymptotic or uniform simultaneous approximation by linear forms or by rational numbers}
\label{SS:SDAHD-FourExponents}

   Let  $ \xi_1,\, \dots, \, \xi_n$ be  real numbers. Assume that the numbers $1,  \xi_1,\, \dots, \, \xi_n$ are linearly independent over $\bQ$.  There are (at least) two points of view for studying  approximation to  \m{ \xi_1,\, \dots, \, \xi_n}.  On the one hand, one may consider  linear forms  (see for instance   \cite{MR33:1286})
      \M{
   |x_0+x_1\xi_1+\cdots+x_n\xi_n|.
   }
 On the other hand, one may investigate the existence of    simultaneous approximation by rational numbers  
   \M{
\max_{1\le i\le n} \left|\xi_i -\frac{x_i}{x_0}\right | .
} 
Each of these two points of view has two versions, an asymptotic one (with exponent denoted  \m{\omega})   and a uniform one (with exponent denoted  \m{\omegahat}). 
This gives rise to four  exponents   introduced in \cite{BugeaudLaurentInhom} (see also \cite{LaurentMichel0703146}),  
\M{
\omega(\theta), \quad \omegahat(\theta), \quad \omega(^t\theta), \quad \omegahat(^t\theta),
}
where
      \M{
      \theta=(\xi_1,\ldots,\xi_n) 
    \quad
    \text{and}\quad
    ^t\theta=\left(
\begin{matrix}
\xi_1\\
\vdots\\
\xi_n\\
\end{matrix}
\right) .
      }
We shall recover the situation of \S\S~\ref{SS:PASASN-PASAAAN} and \ref{SS:PASASN-SRARN}  in the special case where  \m{\xi_i=\xi^{i}}, \m{1\le i\le n}: for \m{\theta=(\xi,\xi^2,\ldots,\xi^n)}, 
   \M{
   \omega(\theta)=\omega_n(\xi),
   \quad
   \omegahat(\theta)=\omegahat_n(\xi),
     \quad
   \omega(^t\theta)=\omega'_n(\xi),
     \quad
   \omegahat(^t\theta)=\omegahat'_n(\xi).
   }
   Notice that the index $n$ is implicit in the notations involving $\omega$, since it is the number of components of $\theta$.

We start with the question of {\it asymptotic   approximation by linear forms}. We denote by \m{\omega(\theta)} the supremum of the real numbers $w$ for which  there exist    {infinitely many} positive integers  \m{N} for which the system 
\begin{equation}\label{E:OmegaTheta}
   |x_0+x_1\xi_1+\cdots+x_n\xi_n|\le N^{-w},
   \quad 
  0<\max_{0\le i\le n} |x_i| \le N,
\end{equation}
  has a solution in rational integers \m{x_0,x_1,\ldots,x_n}. An upper bound for   \m{ {\omega}(\theta)} is a {\it linear independence measure} for 
  \m{1,\xi_1,\ldots,\xi_n}. 
   
The hat version of $\omega(\theta)$ is, as expected,  related to the study of {\it uniform   approximation by linear forms}: we denote by \m{\omegahat(\theta)} the supremum of the real numbers $w$ such that, 
   {for any sufficiently large  integer  \m{N}}, the same  system (\ref{E:OmegaTheta}) has a solution.
  
Obviously  \m{ \omegahat (\theta) \le \omega(\theta)}.

   The second question is that of asymptotic simultaneous  approximation by rational numbers. Following again \cite{LaurentMichel0703146}, we denote by \m{\omega(^t\theta)} the supremum of the real numbers $w$  for which  there exist    {infinitely many} positive integers  \m{N} for which the system 
\begin{equation}\label{E:omegatranspose}
\max_{1\le i\le n} |x_i-x_0\xi_i |\le N^{-w},
   \quad 
    \text{with}\quad
  0<\max_{0\le i\le n} |x_i| \le N
\end{equation}
  has a solution in rational integers \m{x_0,x_1,\ldots,x_n}.  An upper bound for   \m{{\omega}(^t\theta)}  is a {\it simultaneous approximation measure} for 
  \m{1,\xi_1,\ldots,\xi_n}. 
      
The uniform simultaneous  approximation by rational numbers is measured by the hat version of $\omega$: we denote by   \m{\omegahat(^t\theta)} the supremum of the real numbers $w$ such that, for    any sufficiently large  integer  \m{N}, the same  system (\ref{E:omegatranspose}) has a solution.

Again \m{ \omegahat (^t\theta) \le \omega(^t\theta)}.

Transference principles provide relations between $\omega(\theta)$ and  $ \omega(^t\theta)$. The next result (Khintchine, 1929
 \cite{JFM52.0183.01}) shows that $\omega(\theta)=n$ if and only if  $ \omega(^t\theta)=1/n$.

\begin{thm}[Khintchine     transference principle]\label{T:KhinchineTransfer}
If we set  $\omega=\omega(\theta)$ and $^t\omega=\omega(^t\theta)$, then we have
$$
\omega  \ge n \; ^t\omega + n-1
\quad\text{and} \quad 
^t\omega  \ge
\displaystyle \frac{ \omega }{ (n-1) \omega +n}\cdotp
$$
\end{thm}

In order to study these numbers, M.~Laurent introduces in \cite{LaurentMichel0703146}  further exponents as follows.

\subsection{Further exponents, following M.~Laurent}
\label{SS:SDAHD-FE}

For each \m{d} in the range \m{0\le d\le n-1},  M.~Laurent
\cite{LaurentMichel0703146}
 introduces two exponents, one for  {asymptotic} approximation \m{\omega_d(\theta)} and one for {uniform} approximation \m{\omegahat_d(\theta)}, which measures the quality of simultaneous approximation to the given tuple $\theta=(\xi_1,\ldots,\xi_n)$ from various points of view. First embed $\bR^n$ into $\bP^n(\bR)$ by mapping $ \theta= (\xi_1,\ldots,\xi_n)$ to $(\xi_1: \cdots : \xi_n: 1)$. 
 
Now for $0  \le d \le n-1$, define 
\begin{align}\notag  
\omega_d(\theta)
= \sup\Big\{w   ;   \; \text{there exist  infinitely many vectors } 
\hskip 2 true cm
\hfill
\\
\notag
\hfill
\hskip 2 true cm \;
 X = x_0\wedge \dots \wedge x_d  \in \Lambda^{d+1}(\bZ^{n+1}) 
 \; \text{ for which}  \; 
 | X \wedge \theta | \le | X |^{-w} \Big\}
\end{align}
  and   
\begin{align}\notag   
\omegahat_d(\theta)
= \sup\Big\{ w  ;  \;
 \text{for any  sufficiently large \m{N}, there exists } 
  \hskip 2 true cm
  \hfill
\\
\notag
\hfill
 \; 
 X = x_0\wedge \dots \wedge x_d \in \Lambda^{d+1}(\bZ^{n+1}) \; 
  \text{such that}    \; 
   0 < | X | \le N  \quad\text{and}\quad  
 | X \wedge \theta | \le N^{-w} \Big\}.
\end{align}
Hence \m{\omega_d(\theta)\ge \omegahat_d(\theta)}.

The multivector \m{X =x_0\wedge \dots \wedge x_d} is a system of Plücker coordinates of the linear projective subvariety  $L= \langle x_0, \dots , x_d \rangle \subset \bP^n(\bR)$.
Then
\M{
\frac{ | X \wedge \theta | }{ | X | | \theta |}  
}
is essentially the  distance $d(\theta, L) = \min_{ x\in L} d(\theta,x)$
between the image of $\theta$ in    $\bP^n(\bR)$ to $L$.
As a consequence, equivalent definitions are as follows, where   $\rmH(L)$ denotes the Weil height of any system of Plücker coordinates of $L$. 
\begin{align}\notag   
\omega_d(\theta)
= \sup\Big\{ w ;  \; 
 \text{there exist  infinitely many \m{L}, rational over $\bQ$,} 
 \hskip 2 true cm
\\
\notag
\hskip 2 true cm
\dim L =d
\; \text{ and }\; d(\theta , L)  \le \rmH(L)^{-w-1} \Big\}
\end{align}
and
\begin{align}\notag    
\omegahat_d(\theta)
= \sup\Big\{ w  ;   \; 
 \text{for any sufficiently large \m{N}, there exists   \m{L}, rational over $\bQ$, }
 \hfill
\\
\notag
\hfill
 \dim L =d,  \; 
  \rmH(L)  \le N \; \text{and}\;
 d(\theta , L)  \le \rmH(L)^{-1} N^{-w} \Big\}.
\end{align}
In the extremal cases \m{d=0} and \m{d=n-1}, one recovers the exponents of \S~\ref{SS:SDAHD-FourExponents}: 
  $$
\omega_0(\theta)=\omega(^t\theta),\quad
\omegahat_0(\theta)=\omegahat(^t\theta),
\quad
\omega_{n-1}(\theta)  =
\omega(\theta),  \quad
\quad
\omegahat_{n-1}(\theta)  =
\omegahat(\theta).
$$
The lower bound
 \M{
\omegahat_d(\theta) \ge \frac{d+1}{n-d}\quad\text{ for all} \quad d=0,\ldots , n-1
}
 valid for all \m{\theta} (with $1,  \xi_1,\, \dots, \, \xi_n$ linearly independent over $\bQ$) follows from the results of W.M.~Schmidt  in his foundational paper 
\cite{MR0213301}
 (see \cite{BugeaudLaurentTransfer}). 
 In particular for  \m{d=n-1}  and \m{d=0} respectively, this lower bound yields
\M{
\omegahat(\theta)  \ge n 
\quad\text{and} \quad
\omegahat(^t\theta)  \ge 1/ n 
}
and in the special case $\xi_i=\xi^i$ ($1\le i\le n$)  one recovers the lower bounds 
\M{
 \omegahat_n(\xi) \ge n 
\quad\text{and}  \quad
 \omegahat'_n(\xi) \ge 1/ n ,
}
which we deduced in \S~\ref{SS:PASASN-PACN} and \S~\ref{SS:PASASN-SRARN} respectively from Dirichlet's box principle. 

It was proved by Khintchine in 1926 \cite{JFM52.0183.01} that    \m{\omega(\theta)=n} if and only if \m{\omega(^t\theta)=1/n}. In \cite{LaurentMichel0703146}, M.~Laurent slightly improves on earlier inequalities due to W.M.~Schmidt \cite{MR0213301}  splitting the classical Khintchine's  transference principle (Theorem \ref{T:KhinchineTransfer}) into intermediate steps.

 \begin{thm}[Schmidt, Laurent]\label{T:MichelLaurent}
Fix $n\ge 1$ and $\theta= (\xi_1,\ldots,\xi_n) \in\bR^n$. Set
$\omega_d =\omega_d (\theta)$,  $0 \le d \le n-1$.
The ``going up  transference principle'' is
\M{
\omega_{d+1} \ge \frac{ (n-d)\omega_d +1 }{ n-d-1}\virgule \quad 0 \le d \le n-2,
}
while the ``going down   transference principle'' is
\M{
\omega_{d-1} \ge \frac{ d \omega_d }{ \omega_d +d+1}\virgule  \quad  1 \le  d \le  n-1.
} 
Moreover  these estimates are optimal.
\end{thm}

 As a consequence of Theorem \ref{T:MichelLaurent}, one deduces that if $\omega_d=(d+1)/(n-d)$ for one value of $d$ in the range $0\le d\le n-1$, then the same equality holds for all $d=0,1,\ldots,n-1$. Hence, for almost all \m{\theta\in\bR^n},
\M{
\omega_d(\theta)=
\omegahat_d(\theta) =\frac{d+1}{n-d} \quad \text{for} \quad 0\le d\le n-1.
}
 A complement to  Theorem \ref{T:MichelLaurent},
 involving the hat coefficients,  
 is given in  \cite{LaurentMichel0703146}  Th.\ 3.
 
 A problem raised in \cite{LaurentMichel0703146} is to {\it find the spectrum in $(\bR\cup\{+\infty\})^n$ of the $n$-tuples 
 $$
 \bigl(\omega_0(\theta),\dots,\omega_{n-1}(\theta) \bigr),
 $$
 where $\theta$ ranges over the elements $(\xi_1,\ldots,\xi_n)$ in $\bR^n$ with $1,\xi_1,\ldots,\xi_n$ linearly independent over $\bQ$. } 
 Partial results are given in  \cite{LaurentMichel0703146}.
 
 In \cite{BugeaudLaurentInhom} Y.~Bugeaud and M.~Laurent define and study exponents of {\it inhomogeneous} Diophantine approximation. Further progress on this topic has been achieved by M.~Laurent in \cite{LaurentInhomogeneousHausdorff}.

\subsection{Dimension $2$}
\label{SS:D2}

We consider the special case $n=2$ of \S~\ref{SS:SDAHD-FE}: we replace $(\xi_1,\xi_2)$ by $(\xi,\eta)$. So    
let $\xi$ and \m{\eta} be two real numbers with   \m{1,\xi,\eta}  linearly independent over  $\bQ$.

Khintchine's  transference Theorem  \ref{T:KhinchineTransfer} reads in this special case
\M{
\frac{ \omega(\xi,\eta)}{ \omega(\xi, \eta)+2} \le \omega\left(\begin{matrix}\xi
\\
 \eta
\\
 \end{matrix}
 \right)
\le \frac{ \omega(\xi,\eta) -1}{ 2}\cdotp
}
V.~Jarn\'\i k  studied these numbers in a series of papers from 1938 to 1959 (see 
\cite{MR2136100,BugeaudLaurentInhom,LaurentMichelExpDimTwo}). He proved that 
both sides are optimal. Also 
Jarn\'\i k's formula (of which  (\ref{E:JarnikDim2}) is a special case)  reads:
\begin{equation}\label{E:Jarnik}
 \omegahat\left(\begin{matrix}\xi
\\
 \eta
 \end{matrix}
 \right) =  1 -\frac{1}{ \omegahat(\xi,\eta)}\cdotp
 \end{equation}
The spectrum  of each of our four exponents is as follows:  
 \\
\null\qquad
   \m{\omega(\xi,\eta) }  takes any value in the range  \m{[2,+\infty]},
 \\
\null\qquad
    \m{ \omega\left(\begin{matrix}\xi
\\
 \eta
\\
 \end{matrix}
 \right) }  takes any value in the range 
   \m{[1/2,1]},
 \\
\null\qquad
  \m{\omegahat(\xi,\eta) }  takes any value in the range  \m{[2,+\infty]},
   \\
\null\qquad
     \m{ \omegahat\left(\begin{matrix}\xi
\\
 \eta
\\
 \end{matrix}
 \right) }  takes any value in the range 
   \m{[1/2,1]}.
 
 Moreover, for almost all \m{(\xi,\eta)\in\bR^2}, 
   \M{
\omega(\xi,\eta) = \omegahat(\xi,\eta) =2,
\quad
\omega\left(\begin{matrix}\xi
\\
 \eta
\\
 \end{matrix}
 \right) 
 =
  \omegahat\left(\begin{matrix}\xi
\\
 \eta
\\
 \end{matrix}
 \right) 
 =\frac{1}{2}\cdotp
}  
A more precise description of the spectrum of the quadruple is due to M.~Laurent \cite{LaurentMichelExpDimTwo}:

\begin{thm}[Laurent]
Assume  \m{1,\xi , \eta}  are linearly independent over  $\bQ$.  The four exponents
\M{
\omega= \omega(\xi,\eta),\quad  \omega'= \omega\left(\begin{matrix}\xi
\\
 \eta
\\
 \end{matrix}
 \right),
\quad  \omegahat = \omegahat(\xi,\eta), \quad \omegahat'=  \omegahat\left(\begin{matrix}\xi
\\
 \eta
\\
 \end{matrix}
 \right)  
}
are related by  
\M{
2\le \omegahat\le +\infty ,\quad  \omegahat' =\frac{\omegahat-1}{ \omegahat}\virgule 
\quad \frac{\omega(\omegahat-1)}{ \omega+\omegahat} \le \omega' \le \frac {\omega-\omegahat+1}{ \omegahat} 
}
with the obvious interpretation if   \m{\omega=+\infty}.
Conversely, for any  \m{(\omega,\omega',\omegahat,\omegahat')}  in \m{(\bR_{>0}\cup\{+\infty\})^4} 
satisfying the previous inequalities,
there exists   \m{(\xi,\eta)\in \bR^2}, with   \m{1,\xi , \eta} linearly independent over $\bQ$,
for which 
$$
\omega= \omega(\xi,\eta),\quad  \omega'= \omega
\left(\begin{matrix}\xi
\\
 \eta
\\
 \end{matrix}
 \right),
\quad  \omegahat = \omegahat(\xi,\eta), \quad \omegahat'=  \omegahat\left(
\begin{matrix}\xi
\\
 \eta
\\
 \end{matrix}
 \right) . 
$$
\end{thm}

As a consequence:

  \begin{cor}
The   exponents \m{\omega= \omega(\xi,\eta)}, \m{  \omegahat=  \omegahat(\xi,\eta)}
are related by
\M{
\omega \ge \omegahat(\omegahat-1) 
\quad\text{and} \quad
 \omegahat\ge 2.
}
Conversely, for any   \m{(\omega,\omegahat)} satisfying these conditions, there exists  \m{(\xi,\eta)} for which  
\M{
\omega(\xi, \eta)=\omega
\quad\text{and}\quad
\omegahat(\xi,\eta)  = \omegahat.
}
\end{cor}

  \begin{cor}
The  exponents 
\m{ \omega'= \omega\left(\begin{matrix}\xi
\\
 \eta
\\
 \end{matrix}
 \right) }, \m{ \omegahat'=  \omegahat\left(\begin{matrix}\xi
\\
 \eta
\\
 \end{matrix}
 \right)  
}
are related by
\M{
\omega' \ge \frac{\omegahat'^2}{ 1-\omegahat'}  
\quad\text{and} \quad
 \frac{1}{ 2} \le \omegahat' \le  1.
}
Conversely, for any   \m{(\omega',\omegahat')} satisfying these conditions, there exists  \m{(\xi,\eta)} with
\M{
\omega\left(\begin{matrix}\xi
\\
 \eta
\\
 \end{matrix}
 \right)= \omega'
 \quad\text{and}\quad
 \omegahat\left(\begin{matrix}\xi
\\
 \eta
\\
 \end{matrix}
 \right) = \omegahat'.
}

\end{cor}


The next open problem has been raised by M.~Laurent: 


\begin{Pb}[Laurent]

 Is there an extension of Jarn\'\i k's  equality (\ref{E:Jarnik})
in higher dimensions relating \m{ \omegahat(\theta)}  and  \m{ \omegahat(^t \theta)} for \m{\theta\in\bR^n}? 
\end{Pb}

\subsection{Approximation by hypersurfaces}
\label{SS:AH}

In dimension  $1$ an irreducible hypersurface is nothing else than a point. The exponents $\omega_n(\xi)$ and their hat companions in \S~\ref{SS:PASASN-PACN} measure $|P(\xi)|$ for $P\in\bZ[X]$,
while $\omega_n^*(\xi)$  of \S~\ref{SS:PASASN-AACN} measure the distance between a point $\xi\in\bC$ and   algebraic numbers $\alpha$. 

A generalisation of these questions in higher dimensions, where $\xi\in\bC^n$, is the study of $|P(\xi)|$ for $P\in\bZ[X_1,\ldots,X_n]$ and of $\min_\alpha |\xi-\alpha|$ where $\alpha$ runs over the set of  zeros of such $P$. As already mentioned   in the introduction of \S~\ref{S:SDAHD},  a lower bound for   $|P(\xi)|$ when the degree of $P$ is fixed is 
nothing else than a 
linear independence measure for $(\xi,\ldots,\xi^n)$. To consider such quantities also when the degree of $P$ varies yields a generalisation of Mahler's classification to several variables, which has been considered by Yu Kunrui 
\cite{MR88h:11049}. A generalisation to higher dimensions of both Mahler and Koksma classifications has been achieved by W.M.~Schmidt in \cite{SchmidtFunctionesApproximatio2006} who raises a number of open problems suggesting that the close connection between the two classifications in dimension $1$ does not extend to the classification of tuples.

In  \cite{SchmidtTransactions2007}
W.M.~Schmidt  deals with approximation to points $\xi$  in $\bR\sp{n}$ or  in $\bC\sp{n}$ by algebraic hypersurfaces, and more generally by algebraic varieties, defined over the rationals. 
\par
Let ${\mathcal M}$ be a nonempty finite set of monomials in $x\sb{1},\ldots,x\sb{n}$ with $|{\mathcal M}|$ elements. Denote by ${\mathcal P}({\mathcal M})$ the set of polynomials in $\bZ[x\sb{1},\ldots,x\sb{n}]$ which are linear combinations of monomials in ${\mathcal M}$. Using Dirichlet's box principle or Minkowski's theorem on linear forms, one shows the existence of nonzero elements in ${\mathcal P}({\mathcal M})$ for which  $|P(\xi)|$ is small. It is a much more difficult task to get the existence of nonzero elements in ${\mathcal P}({\mathcal M})$ for which  the distance $\delta\bigl(\xi,A(P)\bigr)$ between $\xi$ and the hypersurface $A(P)$ defined by $P=0$ is small. 
\par
W.M.~Schmidt  asks whether given $\xi$ and ${\mathcal M}$, there exists $c=c(\xi,{\mathcal M})>0$ such that there are infinitely many $P\in{\mathcal P}({\mathcal M})$ with 
$\delta\bigl(\xi,A(P)\bigr)\le c H(P)\sp{-m}$, where $m=|{\mathcal M}|$ in the real case $\xi\in\bR\sp n$ and $m=|{\mathcal M}|/2$ in the complex case $\xi\in\bC\sp n$.
He proves such an estimate when $|{\mathcal M}|=n+1$, and also in the real case when $|{\mathcal M}|=n+2$. In the  case $|{\mathcal M}|=n+1$ he proves a uniform result, in the sense of Y.~Bugeaud and M.~Laurent
\cite{MR2149403}: 
given $N\ge 1$, there is a $P\in{\mathcal P}({\mathcal M})$ with height $H(P)\le N$ for which 
$\delta\bigl(\xi,A(P)\bigr)\le c N\sp{1-m} H(P)\sp{-1}$. A number of further results are proved in which the exponent is not the conjectured one. The author also investigates the approximation by algebraic hypersurfaces (another reference on this topic is \cite{MR2002d:11091}). 
\par
Special cases of the very general and deep results of this paper were due to F.~Amoroso, W.D.~Brownawell, M.~Laurent and D.~Roy, P.~Philippon. Further previous results related with Wirsing's conjecture were also achieved by V.I.~Bernik and K.I.~Tishchenko. 
\par
An upper bound for the distance $\delta(\xi,A)$ means that there is a point on the hypersurface $A$ (or more generally the variety $A$) close to $\xi$. The author also investigates the ``size'' of the set of such elements. The auxiliary results  proved in  \cite{SchmidtTransactions2007}  on this question have independent interest.

\subsection{Further metrical results}\label{SS:FurtherMetricalResults}

The answer to the question of Schmidt on approximation of points
$\xi\in\R^m$ by algebraic hypersurfaces $A(P)$ is almost surely
affirmative. This follows from a general theorem established by
Beresnevich, Bernik, Kleinbock and Margulis. With reference to
section 3.5, let $\mathcal{M}$ be a set of monomials of cardinality
$m=|\mathcal{M}|$ in variables $x_1,\dots,x_k$, where we naturally
assume that $m\ge 2$. Further, let $P(\mathcal{M})$ be the set of
polynomials in $\Z[x_1,\dots,x_k]$ which are linear combinations of
monomials in $\mathcal{M}$. Given a function
  $\Psi:\N\to(0,+\infty)$, let
$$
\begin{array}{ccr}
  \cA_k(\Psi,\mathcal{M}) & = &
  \Big\{(\xi_1,\dots,\xi_k)\in[0,1]^k:
|P(\xi_1,\dots,\xi_k)|<H(P)^{-m+2}\Psi(H(P))\quad  \\[1ex]
  &  & \text{ for infinitely many } P\in P(\mathcal{M})\Big\}\,.
\end{array}
$$
We are interested in $|\cA_k(\Psi,\mathcal{M})|$, the
$k$-dimensional Lebesgue measure of $\cA_k(\Psi,\mathcal{M})$.

\begin{thm}[Beresnevich, Bernik, Kleinbock and Margulis]\label{bbkm}
  For any decreasing $\Psi$,
$$
|\cA_k(\Psi,\mathcal{M})|=\left\{\begin{array}{ccc}
                      0  & \text{if} & \sum_{h=1}^\infty
                      \Psi(h)<\infty\,,\\[2ex]
                      1  & \text{if} & \sum_{h=1}^\infty \Psi(h)=\infty\,.
                     \end{array}
\right.
$$
\end{thm}
The convergence case of this theorem has been independently
established by Beresnevich \cite{Beresnevich-02:MR1905790} and Bernik,
Kleinbock and Margulis \cite{Bernik-Kleinbock-Margulis-01:MR1829381}
using different techniques. The multiplicative analogue of the
convergence part of Theorem~\ref{bbkm}, where $H(P)$ is replaced with
$H^\times(P)$, has also been obtained in
\cite{Bernik-Kleinbock-Margulis-01:MR1829381}. In addition,
Theorem~\ref{bbkm} holds when $\mathcal{M}$ is a set of $m$ analytic
functions defined on $(0,1)^k$ and linearly independent over $\R$. The
analyticity assumption can also be relaxed towards a non--degeneracy
condition.

The divergence case is established in
\cite{Beresnevich-Bernik-Kleinbock-Margulis-02:MR1944505} in the
following stronger form connected with \S\ref{SS:AH}, 
where the notation
$\delta$ and $A(P)$ are explained.

\begin{thm}[Beresnevich, Bernik, Kleinbock and Margulis]\label{bbkm2}
Let $\Psi$ be decreasing and such that $\sum_{h=1}^\infty \Psi(h)$ diverges. Then   for almost
  all $\xi=(\xi_1,\dots,\xi_k)\in[0,1]^k$,
$$
    \delta(\xi,A(P))<H(P)^{-m+1}\Psi(H(P))
$$
has infinitely many solutions $P\in P(\mathcal{M})$.
\end{thm}

Taking $\Psi(h)=h^{-1}\log^{-1}h$, we get the following corollary
which answers Schmidt's question in  \S\ref{SS:AH} 
 in the affirmative for
almost all points:

\begin{corollary}
  For almost all $\xi=(\xi_1,\dots,\xi_k)\in\R^k$, the inequality
\begin{equation}\label{c}
    \delta(\xi,A(P))<H(P)^{-m}\log^{-1}H(P)
\end{equation}
has infinitely many solutions $P\in P(\mathcal{M})$.
\end{corollary}

Another interesting corollary corresponds to the special case of
$\mathcal{M}$ being the set of all monomials of degree at most $d$.
In this case we simply have the case of approximation by
multivariable polynomials of degree at most $d$, where now
$$
\displaystyle m=\binom{k+d}{d}.
$$
In the case of convergence in Theorem~\ref{bbkm}, a lower bound for
the Hausdorff dimension of $\cA_k(\Psi,\mathcal{M})$ is implied by a
general theorem for manifolds of Dickinson and Dodson
\cite{DickinsonDodson-2000a}. Obtaining the corresponding upper
bound in general remains an open problem but see
Theorem~\ref{bernik2} and \cite{Baker-1978,
Beresnevich-Bernik-Dodson-02:MR2069553, DodsonRynneVickers-1989b}.
The Hausdorff measure version of Theorem~\ref{bbkm2} has been
established in \cite{Beresnevich-Dickinson-Velani-06:MR2184760}.

\bigskip

Yet another class of interesting problems concerns the measure of
transcendence and algebraic independence of numbers. Recall that
complex numbers $z_1,\dots,z_m$ are called algebraically independent
if,  for any non--zero polynomial $P\in\mathbb{Z}[x_1,\dots,x_m]$, the value 
$P(z_1,\dots,z_m) $ is not $0$.  Actually, $P(z_1,\dots,z_m)$ can still
get very small when $z_1,\dots,z_m$ are algebraically
independent. Indeed, using Dirichlet's Pigeonhole principle, one can
readily show that there is a constant $c_1>0$ such that for any real
numbers $x_1,\dots,x_m$, there are infinitely many polynomials
$P\in\Z[x_1,\dots,x_m]$ such that
\begin{equation}\label{e8}
|P(x_1,\dots,x_m)|<e^{-c_1 t(P)^{m+1}}\,,
\end{equation}
where $t(P)=\deg P+\log H(P)$ is called the type of $P$. A conjecture
of Mahler \cite{Mahler-71:MR0296029} proved by Nesterenko
\cite{Nesterenko-74:MR0347747} says that in the case $m=1$ for almost
all real numbers $x_1$ there is a constant $c_0>0$ such that
$|P(x_1)|>e^{-c_0 t(P)^2}$ for all non--zero $P\in\Z[x]$.  Nesterenko
has also shown that for $\tau=m+2$, for almost all
$(x_1,\dots,x_m)\in\R^m$ there is a constant $c_0>0$ such that
\begin{equation}\label{e9}
|P(x_1,\dots,x_m)|>e^{-c_0 t(P)^\tau}\qquad\text{for all non--zero
}P\in\Z[x_1,\dots,x_m]
\end{equation}
and conjectured that the latter is indeed true with the exponent
$\tau=m+1$. This has been verified by Amoroso
\cite{Amoroso-90:MR1096348} over $\C$ but the `real' conjecture has
been recently established by Nesterenko's student Mikhailov
\cite{Mikhailov-07:MR2362823}:

\begin{thm}[Mikhailov, 2007]
Let $\tau=m+1$. Then for almost all $(x_1,\dots,x_m)\in\R^m$ there
is a constant $c_0>0$ such that $(\ref{e9})$ holds.
\end{thm}

We conclude by discussing the interaction of metrical, analytic and
other techniques in the question of counting and distribution of
rational points near a given smooth planar curve $\Gamma$. In what
follows we will assume that the curvature of $\Gamma$ is bounded
between two positive constants. Let $N_\Gamma(Q,\delta)$ denote the
number of rational points $(p_1/q,p_2/q)$, where $p_1,p_2,q\in\Z$ with
$0<q\le Q$, within a distance at most $\delta$ from $\Gamma$.

Huxley \cite{Huxley-1994-rational_points} has proved that for any
$\ve>0$, \ $N_\Gamma(Q,\delta)\ll Q^{3+\ve}\delta+Q$. Until recently,
this bound has remained the only non--trivial result.  Furthermore,
very little has been known about the existence of rational points
near planar curves for $\delta<Q^{-3/2}$, that is whether
$N_\Gamma(Q,\delta)>0$ when $\delta<Q^{-3/2}$. An explicit question
of this type motivated by Elkies \cite{Elkies-2000} has been
recently raised by Barry Mazur who asks\,: \textsl{``given a
  smooth curve in the plane, how near to it can a point with rational
  coordinates get and still miss?''}  (Question~(3) in
\cite[\S\,11]{Mazur-04:MR2058289}). When $\delta=o(Q^{-2})$, the
rational points in question cannot miss $\Gamma$ if $\Gamma$ is a
rational quadratic curve in the plane (see
\cite{MR1727177}). This leads to $N_\Gamma(Q,\delta)$
vanishing for some choices of $\Gamma$ when $\delta=o(Q^{-2})$. For
example, the curve $\Gamma$ given by $x^2+y^2=3$ has no rational
points~\cite{Beresnevich-Dickinson-Velani-07}.  When $\delta\gg
Q^{-2}$, a lower bound on $N_\Gamma(Q,\delta)$ can be obtained using
Khintchine's transfer principle. However, such a bound
would be far from being close to the heuristic count of $Q^3\delta$
(see \cite{MR1727177}). The first sharp lower bound on
$N_\Gamma(Q,\delta)$ has been given by Beresnevich
\cite{Beresnevich-01:MR1947026}, who has shown that for the parabola
$\Gamma=(x,x^2)$, $N_\Gamma(Q,\delta)\gg Q^3\delta$ when $\delta\gg
Q^{-2}$.

Recently, Beresnevich, Dickinson and Velani
\cite{Beresnevich-Dickinson-Velani-07} have shown that for arbitrary
smooth planar curve $\Gamma$ with non--zero curvature
$N_\Gamma(Q,\delta)\gg Q^3\delta$ when $\delta\gg Q^{-2}$. Moreover,
they show that the rational points in question are uniformly
distributed in the sense that they form a ubiquitous system (see
\cite{Beresnevich-Bernik-Dodson-02:MR1975457} for a discussion on
ubiquity and related notions). They further apply this to get
various metric results about simultaneous approximation to points on
$\Gamma$. These include a Khintchine-type theorem and its Hausdorff
measure analogue. In particular, for any $w\in(1/2,1)$ they
explicitly obtain the Hausdorff dimension of the set of
$w$-approximable points on $\Gamma$:
\begin{align}\label{e11}
\notag
\dim\Big\{(x,y)\in\Gamma:\max\{
\|qx\|, & \|qy\|\}<q^{-w}
\\
&\text{ 
for  infinitely 
many }
 q\in\N\Big\}
=\frac{2-w}{1+w}
\cdotp
\end{align}
Here $\Gamma$ is a smooth planar curve with non--vanishing curvature.
Using analytic methods, Vaughan and Velani \cite{Vaughan-Velani-2006}
have shown that $\ve>0$ can be removed from Huxley's estimate for
$N_\Gamma(Q,\ve)$. Combining the results of
\cite{Beresnevich-Dickinson-Velani-07} and \cite{Vaughan-Velani-2006}
gives the following natural generalisation of Khintchine's theorem. 

\begin{thm}[Beresnevich, Dickinson, Vaughan, Velani]
Let $\psi:\N\to(0,+\infty)$ be monotonic. Let $\Gamma$ be a
$C^{(3)}$ planar curve of finite length $\ell$ with non--vanishing
curvature and let
\begin{align}\notag
\cA_2(\psi,\Gamma)
&=
\\
\notag
\Big\{(x,y)
&
\in\Gamma:\max\{\|qx\|,\|qy\|\}<\psi(q)\text{
holds for infinitely many }q\in\N\Big\}\,.
\end{align}
Then the arclength\footnote{one-dimensional Lebesgue measure on
$\Gamma$} $|\cA_2(\psi,\Gamma)|$ of $\cA_2(\psi,\Gamma)$
satisfies
$$
|\cA_2(\psi,\Gamma)|=
\left\{\begin{array}{ccc}
                      0  & \text{if} & \sum_{h=1}^\infty
                      h\,\psi(h)<\infty\,,\\[2ex]
                      \ell  & \text{if} & \sum_{h=1}^\infty h\,\psi(h)
=\infty\,.
                     \end{array}
\right.
$$
Furthermore, let $s\in(0,1)$ and let $\mathcal{H}^s$ denote the
$s$-dimensional Hausdorff measure. Then
\begin{equation}\label{e12}
\mathcal{H}^s(\cA_2(\psi,\Gamma))=\left\{\begin{array}{ccc}
                      0  & \text{if} \;  & \sum_{h=1}^\infty
                      h^{2-s}\psi(h)^s<\infty\,,\\[2ex]
                      +\infty  & \text{if} \; & \sum_{h=1}^\infty h^{2-s}\psi(h)^s=\infty\,.
                     \end{array}
\right.
\end{equation}

\end{thm}
Note that (\ref{e11}) is a consequence of (\ref{e12}).

In higher dimensions, Dru\c{t}u \cite{Drutu-05:MR2195121} has studied the
distribution of rational points on non--degenerate rational quadrics
in $\R^n$ and obtained a result similar to (\ref{e12}) in the case
when $\psi(q)=o(q^{-2})$. However, simultaneous Diophantine 
approximation on manifolds as well as the
distribution of rational points near manifolds (in particular
algebraic varieties) is little understood. In other words, the
higher-dimensional version of the `near-misses' question of Mazur
mentioned above has never been systematically considered.


 \def\cprime{$'$}
\providecommand{\bysame}{\leavevmode ---\ }
\providecommand{\og}{``}
\providecommand{\fg}{''}
\providecommand{\smfandname}{\&}
\providecommand{\smfedsname}{\'eds.}
\providecommand{\smfedname}{\'ed.}
\providecommand{\smfmastersthesisname}{M\'emoire}
\providecommand{\smfphdthesisname}{Th\`ese}

 \end{document}